\documentclass[12pt]{amsart}
\usepackage[all]{xy}
\usepackage{amssymb,amsmath,amsthm}
\numberwithin{equation}{section}

\newcommand{\Z}{\mathbb{Z}}

\makeatother
\newtheorem{thm}{Theorem}[section]

\newtheorem{corollary}[thm]{Corollary}

\newtheorem{pro}[thm]{Proposition}
\newtheorem{proposition}[thm]{Proposition}
\newtheorem{lem}[thm]{Lemma}
\newtheorem{lemma}[thm]{Lemma}
\theoremstyle{definition}

\newtheorem{rem}[thm]{Remark}

\numberwithin{equation}{section}

\theoremstyle{definition}
\newtheorem*{ack}{Acknowledgements}

\newcommand{\ben}{\begin{enumerate}}
\newcommand{\een}{\end{enumerate}}
\newcommand{\eit}{\begin{itemize}}
\newcommand{\beq}{\begin{equation}}
\newcommand{\eeq}{\end{equation}}

\newcommand{\Fs}{\mathbb{F}_q^{\times}}

\renewcommand{\leq}{\leqslant}
\renewcommand{\geq}{\geqslant}

\newcommand{\X}{\mathbb{X}}
\newcommand{\Y}{\mathbb{Y}}
\newcommand{\Yn}{\Y(n)}
\newcommand{\ex}{\mathbb{E}}
\newcommand{\E}{\mathcal{E}}
\newcommand{\re}{\textup{Re}}

\newcommand{\pr}{\mathbb{P}}
\newcommand{\f}{f_{\Y}}
\newcommand{\fn}{f_{\Yn}}

\newcommand{\sumh}{\sideset{}{^h}\sum}

\begin{document}

\baselineskip=17pt

\title{Large deviations of sums of random variables}

\dedicatory{To the memory of Jonas Kubilius}

\author{ Andrew Granville}
\address{D{\'e}partment  de Math{\'e}matiques et Statistique,   Universit{\'e} de Montr{\'e}al, CP 6128 succ Centre-Ville, Montr{\'e}al, QC  H3C 3J7, Canada; and  Department of Mathematics, University College London, Gower Street, London WC1E 6BT, England.}
   \email{andrew.granville@umontreal.ca}  

\author{Youness Lamzouri}
\address{Institut \'Elie Cartan de Lorraine, Universit\'e de Lorraine, BP 70239, 54506 Vandoeuvre-l\`es-Nancy Cedex, France}

\email{youness.lamzouri@univ-lorraine.fr}
 
 \thanks{  A.G.~is partially supported by grants from NSERC (Canada), and by 
European Research Council  grant, agreement n$^{\text{o}}$ 670239. }


 \begin{abstract} In this paper, we investigate the large deviations of sums of weighted random variables that are \emph{approximately independent}, generalizing and improving some of the results of Montgomery and Odlyzko. We are motivated by examples arising from number theory, including the sequences $p^{it}$, $\chi(p)$, $\chi_d(p)$, $\lambda_f(p)$, and $\text{Kl}_q(a-n, b)$; where $p$ ranges over the primes, $t$ varies in a large interval, $\chi$ varies among all characters modulo $q$, $\chi_d$ varies over quadratic characters attached to fundamental discriminants $|d|\leq  x$, $\lambda_f(n)$ are the Fourier coefficients of holomorphic cusp forms $f$ of (a large) weight $k$ for the full modular group, and $\text{Kl}_q(a, b)$ are the normalized Kloosterman sums modulo a large prime $q$, where $a, b$ vary in $(\mathbb{F}_q)^{\times}$.    
 \end{abstract}

\maketitle

\newcommand{\cbar}{\overline{\chi}}
\newcommand{\pbar}{\overline{\psi}}
\newcommand{\sumstar}{\sideset\and ^* \to \sum}

\section{Introduction}

In this paper we give estimates for the probability of extremely 
  large deviations of sums of weighted random variables that are  approximately independent. Our results generalize and improve results in a wide variety of questions of interest to analytic number theory (arguably in the spirit of some of Kubilius's work). In this section we will carefully formulate seemingly obscure independence hypotheses on distributions of  sums of weighted random variables  and then state results  for such distributions near to the high end of their feasible range. Then, in the next two sections, we will exhibit the application of our results to many examples that have been considered in the literature.
 
 Let $\mathcal A=\{ 1\leq q_1\leq q_2\leq \ldots \}$ be a given infinite sequence   of real numbers which is fairly well behaved, in that if $x\geq 2\beta'$ then
 \begin{equation}\label{A-distn}
 \mathcal A(x):=\#\{ n: q_n\leq x\} = \alpha \int_{\beta'}^x (\log (t/\beta))^{A-1}  (1+O(1/(\log t)^B))dt ,
\end{equation}
 for some constants $\alpha, \beta>0$ and $A\geq 0$, where $B>\max\{A,1\}$ and $\beta'=(2+\beta)^2$.
 In several   number theory applications,   we are interested in  the distribution of large values of   
\begin{equation}\label{TheModel}
H_{\X}(Q):=\sum_{q_n\leq Q } \frac{\X(n)}{q_n},
\end{equation}
where $Q$ is   large   (perhaps $\infty$), and the $\X(n)$ are real valued random variables that are ``independent enough".
The asymptotic for $ \mathcal A(x)$ implies that the corresponding harmonic sum
\[
H(x):=\sum_{q_n\leq x } \frac{1}{q_n}
\]
diverges as $x\to \infty$, so the sum in \eqref{TheModel} can get arbitrarily large if the $\X(n)$ are supported, say, on $[-1,1]$, and $Q\to \infty$.

 For examples, we might take 
 \begin{itemize}
 \item The $q_n$ to be the   positive integers ($\alpha=1, \alpha\beta=1, A=1$);
  \item The $q_n$ to be the primes  ($\alpha=1,\alpha\beta=1, A=0$); or even 
   \item The $q_n$ to be the absolute values of the zeros of the Riemann zeta-function  ($\alpha=1/2\pi,\alpha\beta=1, A=2$), counting each zero and its conjugate.\footnote{One can work with any individual $L$-function or appropriate families of  $L$-functions since they also satisfy the hypothesis with $A=2$.}
  \end{itemize}
  In each of these cases   $B$ may be taken to be arbitrarily large.
 
 What do we mean by the random variables being ``independent enough''?    Since $1$ and $\pi$ are linearly independent over the rationals, one can deduce that the pair $(e^{it},e^{2\pi it})$ is equidistributed in $\mathbb U^2$ (where $\mathbb U$ is the unit circle) as $t$ varies over a long segment over the real line. That is $e^{it}$ and $e^{2\pi it}$ act more-or-less independently, despite being obviously multiplicatively dependent. However we might want to assume that they appear to be independent on a not-too-long interval, and also with several real numbers (in place of $1$ and $\pi$) that we believe are 
  linearly independent over $\mathbb Q$.
  
Our goal is to determine the probability that the sum in \eqref{TheModel} is $>V$, with $V$ large; that is, to determine
\[
\Phi_Q(V)=\Phi_Q(\X,V):= \pr\left(H_{\X}(Q)>V\right),
\]
the probability that $H_{\X}(Q)>V$.
This will depend on several parameters, most importantly $A$, and some invariants of the probability distributions of our random variables.

 \subsection{Families of random variables}  Our goal is to model a sequence of random variables $\X(n)$ that are    ``independent  enough'' by a  sequence of \emph{independent random variables} $\Yn$ for which we determine the distribution of  large values of $H_{\Y}(Q)$. We then show that the model actually provides the correct probability for $H_{\X}(Q)$ in an appropriate range. To explicitly model the $\X(n)$ by the $\Yn$ we will assume the following explicit connection:
 \medskip

\noindent \textbf{Approximate Independence Hypothesis} AIH$(L,Q)$: \newline
\emph{The random variables $\{\X(n)\}_{q_n\leq Q}$  satisfy AIH$(L,Q)$ if there exist independent random variables $\{\Yn\}_{q_n\leq Q}$ for which 
$$ 
\ex\big(\X(n_1)\cdots \X(n_{\ell})\big)= \ex\big(\Y(n_1)\cdots \Y(n_{\ell})\big) +O\left(e^{-L}\right),
$$
for all $\ell \leq L$  and for all choices of $n_1, \dots, n_{\ell} \leq \mathcal{A}(Q)$.}\footnote{Throughout, $\ex(T)$ denotes the expected value of the random variable $T$.}
 \medskip

 It is convenient to divide the $\Yn$ through by an appropriate constant (which can be incorporated into the $q_n$) and re-centered so that each $\Yn$
 is supported on $[-1, 1]$ with $\ex(\Yn)=0$.\footnote{That is, let $\mu_n=\mathbb E(\Yn)$ and $m_n:= \max |\Yn-\mu_n|$ and then we work with $\Y'(n):=\tfrac{\Yn-\mu_n}{m_n}$, so that
$\sum_{q_n\leq Q} \tfrac{\Y(n)}{q_n}=\sum_{q_n\leq Q} \tfrac{\Y'(n)}{q_n'} + \mu$ where $q_n'= q_n/m_n$ and $\mu:=\sum_{q_n\leq Q} \tfrac{\mu_n}{q_n}$.}
It will also be useful if the $\Yn$ ``converge'' to a well-behaved universal distribution $\Y$ as $n\to \infty$:
\smallskip
 
 
\noindent \textbf{Convergent sequence of Random Variables Hypothesis} CRVH$(\Y)$: \newline
\emph{The independent random variables $\{\Yn\}_{n\geq 1}$   satisfy CRVH$(\Y)$ if 
\begin{itemize}
\item  Each $\Yn$ is supported on $[-1, 1]$ with $\ex(\Yn)=0$;
\item  There exists an absolute constant $c>0$ such that 
\begin{equation}\label{LowerBoundf}
\pr(\Yn>1-1/t) \gg \exp(-ce^{\sqrt{t}}) \text{ for all } t\geq 1,
\end{equation} 
where the implicit constant is absolute;
\item There exist independent  random variables $\{\Y_n\}_{n\geq 1}$, each of which is distributed identically to $\Y$, such that $\Y(n)$ depends on $\Y_n$ with $\Yn=\Y_n+O(\tfrac 1{(\log n)^{2}})$.  
\end{itemize}
 }
 \medskip
 
 One can deduce that if this holds then  $\Y$ is   supported on $[-1, 1]$ with $\ex(\Y)=0$, and satisfies 
\eqref{LowerBoundf}. We now define some invariants associated with $\Y$: Let 
\begin{equation}\label{Thefunctionf}
\f(t):=\begin{cases} \log \ex (e^{t\Y}) & \text{ if } 0\leq t < 1,\\ \log \ex (e^{t\Y})- t & \text{ if }\qquad t \geq 1,\end{cases} \qquad \text{ with } \eta_{\Y}=\int_0^{\infty} \frac{\f(u)}{u^2}du.
\end{equation} 
We will show that this last integral converges in Lemma \ref{LogEx} because of the hypotheses on $\Y$. 
    
 
\subsection{The diagonal sum}

Applying partial summation to the counting function $\mathcal A(\cdot)$, we find that (see Lemma \ref{LogSum} below)
if $y\geq 4\beta'$
\begin{equation}\label{H-value}
 H(y) =
 \begin{cases}  \frac{\alpha}{A}(\log y/\beta)^{A} +C_{\mathcal A} +o_{y\to \infty}(1) & \text { if } A>0, \\
   \alpha \log\log y/\beta  +C_{\mathcal A} +o_{y\to \infty}(1) & \text { if } A=0, \\
 \end{cases}
\end{equation}
 for some constant $C_{\mathcal A}$,  since we assumed that $B>A$. 
   
 If $\Phi_Q(V)>0$ then $V\leq H(Q)$, which implies that  $W(V)\leq (1+o(1)) \log (Q/\beta)$ if we define
 \[
W=W(V):= \begin{cases} \big(\tfrac{A}{\alpha} (V-C_{\mathcal A})\big)^{1/A} & \text{ if } A>0,\\
\exp (  \tfrac{1}{\alpha}(V-C_{\mathcal A}) ) & \text{ if } A=0,
\end{cases}
\]
for $V\geq 1+C_{\mathcal A}$.

 We obtain estimates for the distribution of large values of $H_{\Y}(Q)=\sum_{q_n\leq Q} \Y(n)/q_n$ which are more precise than those of Montgomery and Odlyzko \cite{MoOd}, who obtained upper and lower bounds for $\Phi_Q(\Y, V)$ in the case where the $\Y(n)$ are independent and identically distributed, but considered more general weights $q_n$.

\begin{thm}\label{cor: Main2*}
Let $Q$ be large and suppose that $\{\Yn)\}_{q_n\leq Q}$ is a sequence of independent random variables which satisfy CRVH$(\Y)$. Then
  \[
 \Phi_Q(\Y, V)= \exp\left(-\alpha\beta e^{-\eta_{\Y}-1}   W^{A-1}e^W \left(1+O\left(
\frac{\sqrt{\log W}}{W^{1/2}}+\frac{1}{W^{(B-1)/2}}
\right)\right)\right)
\]
for those values of $V$ for which  
 \[
c\leq  W= W(V)\leq   \log Q -     \log\log Q -     \Theta'
\]
for some suitably large constants $\Theta'>0$ and $c\geq 1$. 
\end{thm}

It is arguably surprising that the specifics of the probability distribution $\Y$ only seem  to affect the constant term $e^{-\eta_{\Y}}$ in $ \Phi_Q(\Y, V)$.
All other   parts of this formula are determined from the asymptotic for $\mathcal A(x)$, except for the constant $C_{\mathcal A}$ (which appears in the definition of $W(V)$) which largely depends on the small $q_n$.

We now give our main result, which is applicable to problems in analytic number theory, by starting with a sequence of random variables $\X(n)$ which can be appropriately approximated by a  sequence of independent random variables $\Yn$.

 \begin{thm}\label{cor: Main2} Let $Q$ be large. Suppose that the random variables $\{\X(n)\}_{q_n\leq Q}$  satisfy AIH$(L,Q)$, so that they are approximated by the independent random variables $\{\Yn\}_{q_n\leq Q}$, with $(\log Q)^{2A^2} \leq L \leq \mathcal A(Q)$. Moreover suppose that the $\{\Yn)\}_{q_n\leq Q}$ satisfy CRVH$(\Y)$. Then
  \[
 \Phi_Q(\X, V)= \exp\left(-\alpha\beta e^{-\eta_{\Y}-1}   W^{A-1}e^W \left(1+O\left(
\frac{\sqrt{\log W}}{W^{1/2}}+\frac{1}{W^{(B-1)/2}}
\right)\right)\right)
\]
where  
 \[
c\leq  W= W(V)\leq   \log L -   A  \log\log Q -   \mathbf{1}_{A=0}\cdot \log\log\log Q - \Theta'
\]
for some suitably large constants $\Theta'>0$ and $c\geq 1$. (Throughout  $\mathbf{1}_{\mathcal{C}}$ is the indicator function for the event $\mathcal{C}$.)
\end{thm}

\begin{rem}
 If AIH$(L,Q)$ holds for $L=\mathcal A(Q)\sim \alpha Q (\log (Q/\beta))^{A-1}$, then the range for $W$ here is
\[
1\ll  W(V)\leq    \log Q -    \log\log Q -    \mathbf{1}_{A=0}\cdot \log\log\log Q -\Theta''  .
  \]
 much like in Theorem \ref{cor: Main2*}.
 
If $A>1$ then  $W=\big(\tfrac{A}{\alpha} V\big)^{1/A} +O(1/V^{1-1/A})$, so we may replace $W$ by $W':=\big(\tfrac{A}{\alpha} V\big)^{1/A}$ in Theorem \ref{cor: Main2} at the cost of an additional error term of $O(1/W'^{A-1})=o(1)$. The result therefore becomes
$\log  \Phi_Q(V) \sim  -\kappa   V e^{W'}/W'$ where $\kappa= \beta A e^{-\eta_{\Y}-1} $.

 If $A=1$ then we are at liberty to select $\beta$ since its contribution can be incorporated into the constant $C_{\mathcal{A}}$ (by changing $C_{\mathcal{A}}$ to $C'_{\mathcal{A}}= C_{\mathcal{A}}- \alpha\log\beta$ in \eqref{H-value}), and it is most convenient to  let $\beta=1$.
 The result therefore becomes 
$\log  \Phi_Q(V) \sim  -\kappa e^{V/\alpha}$
 where $\kappa=\alpha  e^{-\eta_{\Y}-1-C_{\mathcal A}/\alpha} $.
  \end{rem}
 
Under the same assumptions in Theorem \ref{cor: Main2}, we will establish that the sum $H_{\X}(Q)$ converges in distribution to that of some random variable $\mathbb{W}$. We will also prove that the characteristic function of $\mathbb{W}$ is rapidly decreasing, which implies that $\mathbb{W}$ is absolutely continuous and has a uniformly bounded probability distribution function. 
Furthermore, we shall use the Berry-Esseen inequality to bound the rate of convergence. 

\begin{thm}\label{LimitLaw}
Under the same assumptions as Theorem \ref{cor: Main2} there exists an absolutely continuous random variable $\mathbb{W}$ such that $H_{\X}(Q)$ converges in distribution to $\mathbb{W}$ as $Q\to \infty$, and moreover we have 
 $$ \sup_{V\in \mathbb{R}} \big|\Phi_Q(\X, V)- \pr(\mathbb{W}>V)\big| \ll  \frac{H(Q)}{L}+  \frac{(\log Q)^{(A-1)/3}}{Q^{1/3}}.$$
\end{thm}

The plan of the paper is as follows. In the next section we describe some examples of sums of independent random variables $H_{\Y}(Q)$ for which Theorem \ref{cor: Main2*} holds. In section 3, we exhibit sequences of random variable $\X(n)$ which satisfy the hypothesis AIH, and present some applications of our results. In section 4, we describe some estimates for the harmonic sum $H(Q)$ that will be useful in our subsequent work. In section 5 we investigate the cumulant generating function of $H_{\Y}(Q)$. Theorem \ref{cor: Main2*} will be proved in section 6. Finally in section 7, we explore how to use the hypothesis AIH and prove Theorems \ref{cor: Main2} and \ref{LimitLaw}.

The methods that we use may be unsurprising to the experts, and the results that we obtain rely on the unusually  strong hypotheses.
However we justify these remarks by noting that our results are  applicable to a very wide selection of number theory distribution problems, and the
ranges in our Theorems  are longer than in the literature on relevant questions, extending very close to the boundary of what is possible.  Example 3 below highlights the value of these results, improving the range in what is believed about the distribution of the error term in the prime number theorem, certainly a very well studied problem.

\section{Examples of sums of random variables $H_{\Y}(Q)$ where Theorem \ref{cor: Main2*} holds}

\subsection*{Example 1: The integers}
We are interested in the large values of 
\begin{equation} \label{eq: zetasum}
  \sum_{n\geq 1} \frac{\Y(n)}{n} ,
\end{equation}
where $\{\Y(n)\}_{n\geq 1}$ is a sequence of independent random variables satisfying CRVH, and converging in distribution to a random variable $\Y$, which is supported on $[-1, 1]$ with $\ex(\Y)=0$, and for which \eqref{LowerBoundf} holds. An important example is to take $\Y(n)= \text{Re}(\X(n))$ where for each $n$, $\X(n)$ is uniformly distributed on the unit circle. Taking $\alpha=1, \beta=1, A=1$ and $Q\to \infty$ in Theorem \ref{cor: Main2*} we obtain for all real numbers $V>0$
$$ \Phi(\Y,V)= \pr\left(\sum_{n\geq 1} \frac{\Y(n)}{n} >V\right)=\exp\left(-(1 + o(1))e^{V-\gamma-\eta_{\Y}-1} \right),$$
where $\gamma$ is the Euler-Mascheroni constant. 
 
\subsection*{Example 2: The primes}

Let $\{\Y(p)\}_{p}$ be a sequence of independent random variables indexed by the primes, and satisfying CRVH$(\Y)$.  Then, 
taking $\alpha=1, \beta=1, A=0$ and $Q\to \infty$ in Theorem \ref{cor: Main2*} we obtain
$$\Phi(\Y,V)= \pr\left(\sum_{p} \frac{\Y(p)}{p} >V\right)=\exp\left(-  (1+ o(1))   e^{ e^{V-C} - V+C-\eta_{\Y}-1}\right),$$
where
\begin{equation}\label{ConstantPrimes}
C: =\lim_{Q\to\infty}  \left( \sum_{p\leq Q} \frac 1{p} -   \log\log Q   \right)=0.2614972128\dots
\end{equation}
We can re-normalize here and obtain
\[
\pr\left(  \prod_p e^ {\Y(p)/p}  >e^CT\right) =\Phi(\Y, \log T+C) =\exp\left(-(1+ o(1))   \frac{e^{  T  -\eta_{\Y}-1}}T \right).
\]
In particular, this result applies to the sum
\begin{equation} \label{eq: sumX}
\text{Re} \left(\sum_{p \text{ prime} } \frac{\X(p)}{p} \right),
\end{equation}
where each $\X(p)$ is independently uniformly distributed on the unit circle.
However,  $\prod_p e^ {\X(p)/p}$ is not exactly  the same as 
$$L(1,\X):=\prod_{p}\left(1-\frac{\X(p)}{p}\right)^{-1},$$
which is used to model the values of $\zeta(1+it)$ and $L(1, \chi)$, though it differs by no more than a bounded (multiplicative) constant.\footnote{And
similarly $\exp(\text{Re}( \sum_{p \text{ prime} } \frac{\X(p)}{p} ))$ differs from $|L(1,\X)|$ by a bounded (multiplicative) constant.}
Therefore, up to that constant, \eqref{eq: sumX} as each $\X(p)$ is independently uniformly distributed on the unit circle models the distribution of  values of $\log |\zeta(1+it)|$ as $t$ varies, as well 
as the distribution of values of  $\log |L(1,\chi)|$ as we vary over all of the non-principal characters mod $q$ as $q\to\infty$, up to a constant.
Moreover  \eqref{eq: sumX} as each $\X(p)$ is independently uniformly distributed at each of the $m$th roots of unity,  models the distribution of  values of  $\log |L(1,\chi)|$ as we vary over the non-principal characters of order $m$, over moduli $q$ for which $m$ is a possible order, up to a constant.

To deduce an estimate for 
$\pr\left(| L(1,\X)  | >T\right)$, one can show  that when $| L(1,\X) |$ is large then so is \eqref{eq: sumX}, and that this implies  $\X(p)\approx 1$ for all ``small'' $p$, with high probability, and so for large values of $| L(1,\X) |$ we ``expect'' that  
\[
| L(1,\X) | \sim \bigg| \prod_p e^ {\X(p)/p} \bigg| \cdot \prod_p \bigg( 1 -\frac{1}{p} \bigg)^{-1} e^{-1/p}.
\]
The effect is to replace the constant $C$ above by $\gamma\approx 0.5772156649\dots$ to obtain 
\[
\pr\left(   |L(1,\X)| >e^\gamma T\right)   =\exp\left(-  (1+ o(1))   \frac{e^{  T  -\eta_{\X}-1}}T \right),
\]
a result that was  obtained unconditionally by Granville and Soundararajan in \cite{GrSo2} by a different method.  In the next subsection we will see how to use our methods to reprove this unconditionally.

\subsection*{Example 3: Sums over zeros of the Riemann zeta-function} 
The \emph{explicit formula} (\cite{Da}, (9) of ch.17) for 
$\psi(x):=\sum_{n\leq x} \Lambda(n)$ implies that the error term in the weighted prime number theorem is
\[
\frac{\psi(x)-x}{\sqrt{x}} = - \sum_{\rho=\tfrac 12+i\gamma:\ \zeta(\rho)=0}  \frac{ x^{i\gamma}} {\rho}  +o(1),
\]
under the assumption of the Riemann Hypothesis.  If $\zeta(\tfrac 12+i\gamma)=0$ then $\zeta(\tfrac 12-i\gamma)=0$. Otherwise we believe that the $\gamma$ are completely linearly independent, and even that the $x^{i\gamma}$, with $\gamma>0$, are truly independent. This means that we can analyze the value of this sum by replacing each $x^{i\gamma}$ by a random variable which is uniformly distributed on the unit circle. However this sum is not quite in the correct form since the $\rho$ are complex numbers, but we can simply adjust our random variable by an angle $-|\rho|/\rho$. Therefore we can model the error term in the weighted prime number theorem by the values of 
$$2\, \text{Re} \left(  \sum_{n\geq 1}  \frac{\X(\gamma_n)}{|\rho_n|} \right),$$
where $0<\gamma_1<\dots$ are the successive ordinates of the zeros $\tfrac 12+i\gamma_n$ of $\zeta(s)$.
Taking $\alpha=1/2\pi, \alpha\beta=1, A=2$ in Theorem \ref{cor: Main2*} we obtain
 \[
 \pr\left( 2\sum_{n\geq 1} \frac{\text{Re}(\X(n))}{|\rho_n|}>T\right)= \exp\left(-  e^{-\eta_{\X}-1}   (  \sqrt{2\pi T} +O(  \sqrt{\log T} )) \exp\left( \sqrt{2\pi T}\right)  \right).
 \] 
Montgomery and Odlyzko obtained the $\sqrt{2\pi V}$ parts of this formula but were unable to show the asymptotic behaviour of $\log \Phi(V)$
(see (13.48) of \cite{MV}). This probability is around $1/u$ when 
\[
T    = \frac 1{2\pi} (\log\log u-\log\log\log u+\eta_{\X}+1 +o(1))^2.
\]
Montgomery and Odlyzko used their estimates to then predict  the largest error term in the prime number theorem:  
\[
\limsup_{x\to\infty} \frac{\psi(x)-x}{\sqrt{x}(\log\log\log x)^2} = \frac 1{2\pi} \text{ and }
\liminf_{x\to\infty} \frac{\psi(x)-x}{\sqrt{x}(\log\log\log x)^2} = -\frac 1{2\pi}
\]
taking $x=e^u$, which is further supported by calculations, and an alternative perspective in Monach's thesis \cite{Monach}
(see also  (15.25), (15.26) of \cite{MV}).  Our work suggests, more precisely, that one can perhaps find arbitrarily large $x$ with 
\begin{equation}
\psi(x)-x=\pm \frac 1{2\pi} \sqrt{x}(\log\log\log x-\log\log\log\log x+\eta_{\X}+1 +o(1))^2
\end{equation}
 but no larger.

\subsection{A more complicated example: The values of an $L$-function}
Here we have 
\[
\Z(p) = -\log \bigg( 1 - \frac{\Y_p}p \bigg)
\]
where each $\Y_p$ is uniformly distributed on the unit circle, or perhaps uniformly distributed on the $m$th roots of unity for some integer $m\geq 2$. We have to begin by re-centering these random variables. First note that 
\[
\ex(\Z(p)) = \sum_{k\geq 1} \frac{\ex(\Y_p^k)}{kp^k} = \sum_{m|k\geq 1} \frac{1}{kp^k} =\ -\frac 1m \log\bigg( 1-\frac{1}{p^m}\bigg)
\]
writing $m=\infty$ if $\Y_p$ is uniformly distributed on the unit circle.  This implies that
\[
 -\log \bigg( 1 + \frac{1}p \bigg) + \frac 1m \log\bigg( 1-\frac{1}{p^m}\bigg)\leq   \text{Re}(\Z(p))-\ex(\Z(p)) \leq -\log \bigg( 1 - \frac{1}p \bigg) + \frac 1m \log\bigg( 1-\frac{1}{p^m}\bigg).
\]
The upper bound is at least as large as the lower bound in absolute value (with equality when $m=2$) and the upper bound is attained.
Therefore define
\begin{equation}\label{NormalizedYp}
\Y(p):= \frac{\text{Re}(\Z(p))-\ex(\Z(p))}{- \log  ( 1 - \frac{1}p  )- \ex(\Z(p)) }  = \text{Re}(\Y_p)+O\Big( \frac 1p \Big).
\end{equation}
This sequence of random variables exactly satisfies CRVH($\Y$). Now define $q_p$ by
\[
\frac 1{q_p}:=- \log  \bigg( 1 - \frac{1}p   \bigg)- \ex(\Z(p)),
\]
so that $q_p=p+O(1)$ and 
\[
 \sum_{p\leq  Q} \frac{\Y(p)}{q_p} = 
- \sum_{p\leq  Q} \text{Re}\bigg(\log \bigg( 1 - \frac{\Y_p}p \bigg)\bigg) -\frac 1m \log \zeta(m) +O\bigg( \frac 1Q\bigg) .
\]
 This means that $\alpha=\alpha\beta=1, A=0$ with
\[
C_{\mathcal A} = \lim_{y\to\infty}  -\sum_{p\leq y}  \bigg( 1 - \frac{1}p   \bigg)- \sum_{p\leq y} \ex(\Z(p)) -\log\log y=\gamma-\frac 1m \log \zeta(m).
\]
 Therefore by Theorem \ref{cor: Main2*}
we have 
\[
\pr\left( \bigg|   \prod_{p\leq  Q}  \bigg( 1 - \frac{\Y_p}p \bigg)^{-1}\bigg|   >e^{\gamma}W \right) =   \exp\left(-     \frac{e^{W-\eta_{\Y}-1}}W \left(1+O\left(
\frac{\sqrt{\log W}}{W^{1/2}} 
\right)\right)\right)
\]
in the range $1\ll W\leq   \log Q -     \log\log Q -     \Theta'.$   
This corresponds to understanding the distribution of the values of $\zeta(1+it)$ as $t$ varies, as well as the values of $L(1,\chi)$ as $\chi$ varies over all of the characters mod $q$ for a large modulus $q$.

\begin{rem}
This generalizes rather nicely. Fix $m>1$ (including $m=\infty$) and let $C_m$ be the set of $m$th roots of unity.
Let $f_p(t)=t+\sum_{k\geq 2} a_{k,p}t^k\in \mathbb R[[t]]$ with $f_p(0)=0$ which converges absolutely in $|t|\leq \frac 1p$.
Let $\mu_{p,m}=\sum_{m|k\geq 1} a_{k,p}/p^k$, and let $F_p(t)=f_p(t/p)$. Assume that the maximum of $|\text{Re}(F_p(z) -\mu_{p,m})|$ over the $m$th roots occurs at $z=1$. 

Let $\Z(p)=F_p(\Y_p)$, then $1/q_p:=F_p(1)-\mu_{p,m}$ with $\Y(p)=q_p( \text{Re}(\Z(p))-\mu_{p,m})$. Proceeding as above we obtain
 \[
\pr\left( \bigg|   \prod_{p\leq  Q}  e^{\Z(p)}\bigg|   >e^{\gamma_f}W \right) =   \exp\left(-     \frac{e^{W-\eta_{\Y}-1}}W \left(1+O\left(
\frac{\sqrt{\log W}}{W^{1/2}} 
\right)\right)\right)
\]
in the range $1\ll W\leq   \log Q -     \log\log Q -     \Theta'$, where we have 
\[
\prod_{p\leq y} e^{F_p(1)} \sim:  e^{ \gamma_f} \log y.
\]
We can rewrite this in terms of the distance from the maximum:
 \[
\pr\left( \bigg|   \prod_{p\leq  Q}  e^{\Z(p)-F_p(1)}\bigg|   > \frac{W}{\log Q} \right) =   \exp\left(-     \frac{e^{W-\eta_{\Y}-1}}W \left(1+O\left(
\frac{\sqrt{\log W}}{W^{1/2}} 
\right)\right)\right).
\]
This corresponds to understanding the distribution of the values of  $L(1,\chi)$ as $\chi$ varies over   the characters of order $m$ and  conductor up to some large $x$ -- the  main results in \cite{GrSo1}  yield a result like this when $m=2$.
\end{rem}

\section{Examples of sequences satisfying AIH}
 
 \subsection{ The sequence $p^{it}$}  
For $\re(s)>1$ we let 
 $$ G(s)= \sum_{p} \frac{1}{p^s}.$$
 Then, one has 
 $$ \zeta(s) = \exp(G(s))R(s),$$
 where $R$ is analytic in the half plane $\re(s)>1/2$. Therefore, using Lemma 2.2 of \cite{GrSo2}, which is derived using the classical zero density estimates for $\zeta(s)$ (see for example Theorem 9.19 A of \cite{Ti}), we have 
 $$ G(1+it)= \sum_{p\leq (\log T)^{C_0}} \frac{1}{p^{1+it}}+ O\left(\frac{1}{\log T}\right),$$
 for all $t \in [T, 2T]$ except of a set of measure $T^{4/5}$, where $C_0$ is an absolute constant, and $T$ is large. 
 
 We let $\X(p)=\text{Re}(p^{it})$, for $t$ varying from $T$ to $2T$. Then we take $\Y(p)=\text{Re}(\Y_p)$, where each $\Y_p$ is a random variable uniformly distributed on the unit circle.
We will see that this sequence satisfies the Approximate Independence Hypothesis AIH$(L,Q)$ with $Q= (\log T)^{C_0}$ and $L\sim 2(\log T)/(C_0 \log\log T)$.
 
\begin{lem}\label{OrthogonalityPrimes}
Let $\X(p)=\textup{Re}(p^{it})$ and define $\ex_T(f(t))=\tfrac 1T\int_T^{2T} f(t) dt$.
Suppose that $Q=(\log T)^{C_0}$, where $C_0$ is a positive constant.
Select $L$ so that $(e^2Q)^{L/2}\leq T$.
For all primes $p_1, \dots, p_{\ell} \leq Q$ and any $\ell \leq  L$, we have 
$$ \ex_T\big(\X(p_1)\cdots \X(p_{\ell })\big)= \ex\big(\Y(p_1)\cdots \Y(p_{\ell})\big) +O(e^{-L}),$$
where  $\Y(p)= \cos(\theta_p)$ for each prime $p$, and $\{\theta_p\}$ is a sequence of independent random variables uniformly distributed on $[0, 2\pi]$.
\end{lem}

\begin{proof} Let $N=p_1\cdots p_\ell$. The left-hand side above is 
\[
 \frac 1T \int_T^{2T} \bigg(   \prod_{j=1}^\ell \frac{p_j^{it}+p_j^{-it}}2 \bigg)  dt = \frac 1{2^\ell} 
 \sum_{ab=N}
  \frac 1T \int_T^{2T} (a/b)^{it} dt.
\]
The terms with $a=b$ give  exactly   the right-hand side   (as can be seen by expanding the right-hand side in the analogous way),
so the difference between the two sides is
\[
\leq \max_{ab=N, a\ne b}  \bigg| \frac{1}{T} \int_T^{2T} \left(\frac{a}{b}\right)^{it}dt \bigg| =\max_{ab=N, a\ne b} \bigg|  \left[ \frac{(a/b)^{it} } { iT \log (a/b)} \right]_T^{2T} \bigg| 
\ll \max_{ab=N, a\ne b}\frac{1}{T| \log a/b|} .
\]
This is  $\ll \frac 1T$ unless $a,b\asymp \sqrt{N}$, in which case it is $\ll   \frac{b}{T| b-a|} \ll   \frac{\sqrt{N}}{T }  \ll \frac{Q^{\ell/2}}{T } \ll e^{-L} $ as $N\leq Q^\ell$ and  $e^LQ^{L/2}\leq T$.
\end{proof}

Using this result in Theorem \ref{cor: Main2} with $Q=(\log T)^{C_0}$ and $L\sim 2(\log T)/(C \log\log T)$, we deduce that
\begin{equation}\label{MeasureLogZeta}
\frac{1}{T} \text{meas} \left\{t\in [T, 2T] : |\exp(G(1+it))|>e^{C}V\right\}= \exp\left(-(1+ o(1))   \frac{e^{ V  -\eta_{\Y}-1}}V \right),
\end{equation}
for $V$ in the range $1\ll V\leq \log_2T-\log_3T- \log_4T-\Theta$, 
where here and throughout $\log_k$ is the $k$-th iterate of the natural logarithm, $C$ is defined in \eqref{ConstantPrimes}, $\Theta$ is a large constant, and $\Y=\re(\mathbb{W})$  where $\mathbb{W}$ is uniformly distributed on the unit circle. Granville and Soundararajan \cite{GrSo2} proved a similar result (in a slightly bigger range) for $|\zeta(1+it)|$. Note that in accordance with section 2.1, one can obtain the analogue of \eqref{MeasureLogZeta} for $|\zeta(1+it)|$ by proving the analogous AIH hypothesis for the random variables $\X(p)=-\re(\log(1-1/p^{1+it}))/(-\log(1-1/p))$ (instead of $\re(p^{it})$), which follows along the same lines of Lemma \ref{OrthogonalityPrimes}, but is slightly more technical so we omit it. 

\subsection{The sequence $\chi(p)$ for characters  $\chi\pmod q$}  
This sequence is similar to the above example of $p^{it}$'s. We let $\X(p)=\text{Re}(\chi(p))$, for $\chi$ varying among all characters $\chi \pmod q$, where $q$ is a large integer. 
Then we let $\Y(p)=\text{Re}(\Y_p)$ where each $\Y_p$ is a random variable uniformly distributed on the unit circle.

We will see that this sequence satisfies the Approximate Independence Hypothesis AIH$(L,Q)$ with $Q= (\log q)^{C_0}$ and any $L$ such that $Q^L<q$.

\begin{lem}\label{OrthogonalityCharacters}
Let $\X(p)=\textup{Re}(\chi(p))$ and define $\ex_q(f(\chi))=\tfrac{1}{\phi(q)}\sum_{\chi \pmod q} f(\chi)$.
Suppose that $Q=(\log q)^{C_0}$ where $C_0$ is a positive constant.
Select $L$ so that $Q^{L}< q$.
For all primes $p_1, \dots, p_{\ell} \leq Q$ and any $\ell \leq  L$, we have 
\begin{equation}
\label{AIHChar} \ex_q\big(\X(p_1)\cdots \X(p_{\ell })\big)= \ex\big(\Y(p_1)\cdots \Y(p_{\ell})\big),
\end{equation}
where  $\Y(p)= \cos(\theta_p)$ for each prime $p$, and $\{\theta_p\}$ is a sequence of independent random variables uniformly distributed on $[0, 2\pi]$.
\end{lem}

\begin{proof} Let $N=p_1\cdots p_\ell$. The left-hand side of \eqref{AIHChar} is  
\[
 \frac{1}{\phi(q)}\sum_{\chi \pmod q} \bigg(\prod_{j=1}^\ell \frac{\chi(p_j)+\overline{\chi}(p_j)}2 \bigg) = \frac 1{2^\ell} 
 \sum_{ab=N}
  \frac{1}{\phi(q)}\sum_{\chi \pmod q}\chi\left(\frac{a}{b}\right).
\]
As in the proof of Lemma \ref{OrthogonalityPrimes}, the contribution of the diagonal terms with $a=b$ 
equals 
$$ \ex\big(\Y(p_1)\cdots \Y(p_{\ell})\big).$$
Moreover, since $Q^{\ell}<q$ we have $a, b <q$ and hence by the orthogonality relations of the characters, the contribution of the off-diagonal terms $a\neq b$ is $0$.  
\end{proof}
\begin{rem}
We observe that the same proof works to show that the sequence $\chi(p)$ verify AIH$(L,Q)$ under the same assumptions on $Q$ and $L$,  where $\chi$ varies over the set of primitive characters $\pmod q$, and $q$ is a large prime. In this case one should add the error term $O(1/q)$ to the right hand side of \eqref{AIHChar}.

We also remark that one might extend the range $Q^L<q$, by averaging instead over all moduli $q\leq X$ for a large number $X$. 

\end{rem}


As in the previous section, one can apply this result together with Theorem \ref{cor: Main2} to investigate the distribution of $\re(G_{\chi}(1))$ where 
\begin{equation}\label{Gchi}
G_{\chi}(s)=\sum_{p} \frac{\chi(p)}{p^s}, 
\end{equation}
for $\re(s)>1$. Indeed, Lemma 8.2 of \cite{GrSo0} together with the classical zero density estimates for the family of Dirichlet $L$-functions (see for example Theorem 12.1 of \cite{Montgomery}) imply that for a positive constant $C_0$ we have
\begin{equation}\label{TruncGchi}
G_{\chi}(1)= \sum_{p\leq (\log q)^{C_0}} \frac{\chi(p)}{p}+ O\left(\frac{1}{\log q}\right)
\end{equation}
for all non-principal characters modulo a large prime $q$, except for a set of size $\ll \sqrt{q}$. Therefore, combining Lemma \ref{OrthogonalityCharacters} with Theorem \ref{cor: Main2} with $Q=(\log q)^{C_0}$ and $L= [\log q/(C_0\log\log q)]-1$ we deduce that 
\begin{equation}\label{DistribSumCharacters}
 \frac{1}{q-1}\left| \{ \chi\bmod q : |\exp(G_{\chi}(1))|>e^CV\}\right| = \exp\left(-(1+ o(1))   \frac{e^{ V  -\eta_{\Y}-1}}V \right),
\end{equation}
for $V$ in the range $1\ll V\leq \log_2q-\log_3q- \log_4q-\Theta$, 
where $C$ is defined in \eqref{ConstantPrimes}, $\Theta$ is a large constant, and $\Y=\re(\mathbb{W})$  where $\mathbb{W}$ is uniformly distributed on the unit circle. Granville and Soundararajan \cite{GrSo2} proved the analogue of \eqref{DistribSumCharacters} for $|L(1,\chi)|$.

\subsection{Values of the quadratic characters at the primes}  
Here we consider the sequence $\chi_d(p)$ where $\chi_d$ is the Kronecker symbol modulo $d$, and $d$ varies over fundamental discriminants $|d|\leq x$. The associated independent random variables are such that
$\Y(p) = 1$ or $-1$ with probability $\frac p{2(p+1)}$  and $\Y(p) = 0$ with probability $\frac 1{p+1}$.

In this case, using the estimates of Graham and Ringrose \cite{GrRi} on character sums to smooth moduli, we can obtain a non-trivial improvement on the range where the Approximate Independence Hypothesis is valid. Indeed, we will see that this sequence satisfies the AIH$(L,Q)$ with $Q= (\log x)^{C_1}$ and any $L$ such that $Q^L\leq x^{C_2\log_3 x}$, where $C_1$ is an arbitrarily positive constant, and $C_2>0$ is a suitably small constant that depends on $C_1$. Note that by using the P\'olya-Vinogradov or the Burgess inequalities instead, we would obtain a range of the form $Q^L\leq x^{C_3}$ for some positive constant $C_3$. 

\begin{lem}\label{OrthogonalityQuadChar}
Let $x$ be large and $\mathcal{F}(x)$ denote the set of fundamental discriminants $d$ with $|d|\leq x$. Let $\X(p)=\chi_d(p)$ and define $\ex_x(f(d))=\tfrac{1}{|\mathcal{F}(x)|}\sum_{d\in \mathcal{F}(x)} f(d)$.
Suppose that $Q=(\log x)^{C_1}$, for some positive constant $C_1$. 
Select $L$ so that $Q^{L}< x^{C_2\log_3x}$, for some suitably small constant $C_2>0$ that depends on $C_1$. 
For all primes $p_1, \dots, p_{\ell} \leq Q$ and any $\ell \leq  L$, we have 
$$
\ex_x\big(\X(p_1)\cdots \X(p_{\ell })\big)= \ex\big(\Y(p_1)\cdots \Y(p_{\ell})\big) + O(e^{-L}),
$$
where  $\Y(p)$ is a sequence of independent random variables taking the values $1$ or $-1$ with probability $\frac p{2(p+1)}$,  and the value $0$ with probability $\frac 1{p+1}$. 
\end{lem}
To prove this result we need the following lemma.  
\begin{lem}\label{GrahamRingrose}
Let $x$ be large and  $n\geq 1$ be an integer. Suppose that all prime factors of $n$ are below $y$, where $y\leq  x$ is a real number. Let $k\geq 1$ be any integer and put $K=2^k$. Then
$$ \frac{1}{|\mathcal{F}(x)|}\sum_{d\in \mathcal{F}(x)}\chi_d(n)
\begin{cases}
= \prod_{p|n}\left(\frac{p}{p+1}\right)+O\left(x^{-1/4}\right) & \text{ if } n\text{ is a square}, \\
\ll x^{-\frac{k}{8K}}\prod_{p|n}\left(1+\frac{1}{p^{1-k/8K}}\right)y^{1/3}n^{\frac{1}{7K}}\tau(n)^{k^2/K} & \text{ otherwise}.
\end{cases}
$$
where $\tau(n)$ is the divisor function.
\end{lem} 
\begin{proof}
When $n$ is a square, the asymptotic for the character sum is standard since in this case $\sum_{d\in \mathcal{F}}\chi_d(n)$ corresponds to the number of fundamental discriminants $d$ that are coprime to $n$ and such that $|d|\leq x$.  On the other hand, in the case where $n$ is not a square, the stated bound corresponds to Lemma 4.2 of \cite{GrSo1}, which is a consequence of Theorem 5 of Graham-Ringrose \cite{GrRi}. 
\end{proof}
\begin{proof}[Proof of Lemma \ref{OrthogonalityQuadChar}] Let $N=p_1\cdots p_\ell$. First, note that for any prime $p$ and positive integer $a$ we have
$$ \ex(\Y(p)^a)= \begin{cases} \frac{p}{p+1} & \text{ if } a \text{ is even},\\
0 & \text{ if } a \text{ is odd}. 
\end{cases} $$
Hence, we deduce that
$$ \ex\big(\Y(p_1)\cdots \Y(p_{\ell})\big)= \begin{cases} \prod_{p\mid N}\left(\frac{p}{p+1}\right) & \text{ if } N \text{ is a square},\\
0 & \text{ otherwise}. 
\end{cases}$$
Let $k= [\frac{\log_3 x}{\log 2}]$ so that $K=2^k\in (\frac 12 \log_2 x,\log_2 x]$.   Lemma \ref{GrahamRingrose} then implies that 
$$
\ex_x\big(\X(p_1)\cdots \X(p_{\ell })\big)= \frac{1}{|\mathcal{F}(x)|}\sum_{d\in \mathcal{F}(x)}\chi_d(N)=  \ex\big(\Y(p_1)\cdots \Y(p_{\ell})\big) + E_1(N)
$$
where 
$$
 E_1(N) \ll \begin{cases}
x^{-1/4} & \text{ if } N\text{ is a square}, \\
x^{-\frac{k}{8K}}\prod_{p|N}\left(1+\frac{1}{p^{1-k/8K}}\right)Q^{1/3}N^{\frac{1}{7K}}\tau(N)^{k^2/K} & \text{ otherwise}.
\end{cases}
$$
Therefore, if $N$ is a square we obtain the desired bound $ E_1(N)\ll e^{-L}$.  Otherwise
 $\prod_{p|N} (1+\frac{1}{p^{1-k/8K}} )\leq \tau(N)\leq 2^{\ell}\leq 2^L$ and 
$N^{\frac{1}{7K}}\leq Q^{\frac{\ell}{7K}} \leq Q^{\frac{L}{7K}}< x^{\frac{2C_2\log_3x}{7\log_2 x}}$. Therefore
\[
 E_1(N) \ll x^{-\frac{\log_3 x}{(8\log 2) \log_2 x}} 4^L  (\log x)^{C_1/3}x^{\frac{2C_2\log_3x }{7\log_2 x}}  <e^{-L}
\]
since $L\leq  \frac{C_2\log x\log_3x}{C_1\log_2 x}$
provided $(\frac 27+\frac{1+\log 4}{C_1})C_2< \frac 1{8\log 2}$.
\end{proof}

As an application we can deduce from this result together with the analogue of \eqref{TruncGchi} for quadratic characters (which follows from Lemma 8.2 of \cite{GrSo0} together with the zero density result of Heath-Brown \cite{HB} for Dirichlet $L$-functions attached to quadratic characters), and Theorem \ref{cor: Main2} with $Q= (\log x)^{C_1}$ and $L=[C_2\log x\log_3 x/(\log_2x)]$ that 
\begin{equation}\label{DistribSRealCharacters}
 \frac{1}{|\mathcal{F}(x)|}\left| \{d\in \mathcal{F}(x) : \exp(G_{\chi_d}(1))>e^CV\}\right| = \exp\left(-(1+ o(1))   \frac{e^{ V -\eta_{\Y}-1}}V \right),
\end{equation}
for $V$ in the range $1\ll V\leq \log_2x-\log_3x-\Theta$, 
where $G_{\chi}$ is defined in \eqref{Gchi}, $C$ is defined in \eqref{ConstantPrimes}, $\Theta$ is a large constant and $\Y$ takes the values $1$ and $-1$ with equal probability $1/2$. Granville and Soundararajan \cite{GrSo1} proved the analogue of \eqref{DistribSRealCharacters} for $|L(1,\chi_d)|$ in a slightly bigger range for $V$.  Again one can recover their result (in our range of $V$) by modifying the sequence that satisfies the AIH from $\chi_d(p)$ to $\X(p)=-\log (1-\chi_d(p)/p)/(-\log(1-1/p))$. 


\subsection{Fourier coefficients of holomorphic cusp forms at the primes} 
We let $\mathcal{H}_k$ be the set of Hecke eigencuspforms of weight $k$ for
the full modular group SL$(2, \mathbb{Z})$. We will assume that the  weight $k$ is large and note
that $\mathcal{H}_k$ contains about $k/12$ forms. Given $f\in \mathcal{H}_k$ we write its Fourier expansion as
$$ f(z)= \sum_{n=1}^{\infty} \lambda_f(n) n^{\frac{k-1}{2}} e^{2\pi i nz},$$
where we have normalized the Fourier coefficients so that the Hecke eigenvalues $\lambda_f(n)$
satisfy Deligne's bound $|\lambda_f(n)|\leq \tau(n)$. In view of the Petersson trace formula (see \eqref{Petersson} below), it is convenient to use the ``harmonic weights'' 
$$ \omega_f:= \frac{\Gamma(k-1)}{(4\pi)^{k-1}\langle f, f \rangle}, $$ 
where $\langle f,g \rangle$ denotes the Petersson inner product. In particular one has
\begin{equation}\label{AverageHarmonicWeight} 
\sum_{f\in \mathcal{H}_k} \omega_f =1+O(k^{-5/6}),
\end{equation}
and hence the harmonic weight is close to the natural weight $1/|\mathcal{H}_k|$ on average. These facts are standard and may be found in \cite{Iw}. We now define 
$$ \ex_k(g(f)):=\frac{1}{|\mathcal{H}_k|_h} \sumh_{f\in \mathcal{H}_k}  g(f),$$
where 
$$  \sumh_{f\in \mathcal{H}_k} g(f)= \sum_{f\in \mathcal{H}_k} \omega_f g(f), \text{ and } |S|_h= \sumh_{f\in S}1 \text{ for } S\subseteq \mathcal{H}_k.$$
\begin{lem}\label{OrthogonalityFourier}
Let $\X(p)=\lambda_f(p)/2$.
Suppose that $Q=(\log k)^{C_0}$, for some positive constant $C_0$. 
Select $L$ so that $Q^{L}\leq k^{2}/10$. 
For all primes $p_1, \dots, p_{\ell} \leq Q$ and any $\ell \leq  L$, we have 
$$
\ex_k\big(\X(p_1)\cdots \X(p_{\ell })\big)= \ex\big(\Y(p_1)\cdots \Y(p_{\ell})\big) + O(e^{-L}),
$$
where  $\Y(p)=\cos(\theta_p)$, and $\theta_p$ is a sequence of independent random variables distributed according to the Sato-Tate measure $d \mu_{\textup{ST}}=\frac{2}{\pi}\sin^2\theta d\theta$  on $[0, \pi]$. 
\end{lem}

\begin{proof}
Let $p$ be a prime number and $a$ a positive integer. We start by expressing $\lambda_{f}(p)^a$ in terms of $\lambda_f(p^b)$ for $0\leq b\leq a$ (note that $\lambda_f(n)$ is not completely multiplicative).  
By Deligne's bound $|\lambda_f(p)|\leq 2$, we may write $\lambda_f(p)= 2\cos(\theta_f(p))$, for some $\theta_f(p) \in [0, \pi]$. Recall that for any $b\geq 0$ we have that
$ \lambda_f(p^b)=\sin((b+1)\theta_f(p))/\sin\theta_f(p).$
We now use the fact that the functions $\{S_n\}_{n\geq 0}$, defined by 
$$S_n(\theta):=\frac{\sin((n+1)\theta)}{\sin\theta},$$
form an orthonormal basis of $L^2([0,\pi],\mu_{ST})$. This implies 
\begin{equation}\label{Hecke}
\lambda_{f}(p)^a= 2^a \cos(\theta_f(p))^a= 2^a\sum_{0\leq b\leq a}C_a(b) \frac{\sin((b+1)\theta_f(p))}{\sin\theta_f(p)}=2^a\sum_{0\leq b\leq a}C_a(b)\lambda_f(p^b), 
\end{equation}
where the coefficients $C_a(b)$ are defined by
$$
C_a(b):=\frac{2}{\pi}\int_0^{\pi}\cos(\theta)^a\sin\theta\sin((b+1)\theta)d\theta.
$$
Note that \eqref{Hecke} is equivalent to the standard Hecke relations
$$ \lambda_f(m)\lambda_f(n)= \sum_{d\mid(m, n)} \lambda_f\left(\frac{mn}{d^2}\right).$$

We now write $p_1\cdots p_{\ell}= q_1^{a_1}\cdots q_r^{a_r}$ where $q_1, \dots, q_r$ are distinct primes, and $a_1+\cdots + a_r= \ell$. Hence, it follows from \eqref{Hecke} that 
\begin{align*}
\ex_k\big(\X(p_1)\cdots \X(p_{\ell })\big)&=\frac{1}{|\mathcal{H}_k|_h}\sum^h_{f\in \mathcal{H}_k} \prod_{j=1}^r \cos(\theta_f(q_j))^{a_j} \\
& = \frac{1}{|\mathcal{H}_k|_h}\sum^h_{f\in \mathcal{H}_k} \prod_{j=1}^r \bigg(\sum_{0\leq b_j\leq a_j}C_{a_j}(b_j)\lambda_f\left(q_j^{b_j}\right)\bigg)\\
&= \sum_{0\leq b_1\leq a_1}\sum_{0\leq b_2\leq a_2}\cdots \sum_{0\leq b_r\leq a_r} \prod_{j=1}^rC_{a_j}(b_j) \frac{1}{|\mathcal{H}_k|_h}\sum^h_{f\in \mathcal{H}_k}\lambda_f\big(q_1^{b_1} \cdots q_r^{b_r}\big).
\end{align*}
To estimate the inner sum, we use the following version of the Petersson trace formula which follows from Lemma 2.1 of \cite{RuSo} together with \eqref{AverageHarmonicWeight}
\begin{equation}\label{Petersson}
 \frac{1}{|\mathcal{H}_k|_h}\sumh_{f\in \mathcal{H}_k}\lambda_f(m)= \mathbf{1}_{m=1} + O\left(k^{-5/6}\right),
 \end{equation}
for all $m\leq k^2/10$. Since $Q^L\leq k^2/10$, we deduce that 
$$\ex_k\big(\X(p_1)\cdots \X(p_{\ell })\big)= \prod_{j=1}^r C_{a_j}(0) + E_1,$$
where 
$$ E_1\ll k^{-5/6}2^r \prod_{j=1}^{r} (a_j+1)\ll  k^{-5/6} 4^{\ell} \ll e^{-L}, 
$$
since $a+1\leq 2^a$ and $|C_a(b)|\leq 2$ for all $a, b$. 
The result follows upon noting that 
$$ \prod_{j=1}^r C_{a_j}(0)= \prod_{j=1}^r\ex\left(\Y(q_j)^{a_j}\right)= \ex\big(\Y(p_1)\cdots \Y(p_{\ell })\big).$$

As an application we will estimate the distribution of large values of $\re(G_f(1))$, where for $\re(s)>1$
$$ G_f(s)= \sum_{p}\frac{\lambda_f(p)}{p^s}.$$
In this case, Lemma 4.2 of \cite{CoMi} together with the zero density estimates of Kowalski and Michel (see Theorem 4 of \cite{KoMi}) imply that for all $f\in \mathcal{H}_k$ one has 
$$ G_f(1)=\sum_{p\leq (\log k)^{C_0}}\frac{\lambda_f(p)}{p}+O\left(\frac{1}{\log k}\right)$$ except for a set of cardinality $\ll k^{9/10}$, where $C_0$ is a positive constant. Therefore,  it follows from Lemma \ref{OrthogonalityFourier} together with Theorem \ref{cor: Main2} with $Q= (\log k)^{C_0} $ and $L\sim 2\log k/(C_0\log\log k)$ that
\begin{equation}\label{DistribFourierPrimes}
 \frac{1}{|\mathcal{H}_k|_h}\left| \{f\in \mathcal{H}_k : \exp(G_f(1))>(e^CV)^2\}\right|_h = \exp\left(-(1+ o(1))   \frac{e^{ V -\eta_{\Y}-1}}V \right),
\end{equation}
for $V$ in the range $1\ll V\leq \log_2k-\log_3k-\log_4 k-\Theta$, where $C$ is defined in \eqref{ConstantPrimes}, $\Theta$ is a large constant and $\Y= \cos(\theta)$ where $\theta$ is distribued according to the Sato-Tate measure on $[0, \pi]$. Liu, Royer and Wu  \cite{LiRoWu} proved the analogue of \eqref{DistribFourierPrimes} for $|L(1,f)|$, where $L(s, f)$ is the $L$-function attached to the eigencuspform $f$.  One can recover their result  by modifying the sequence that satisfies the AIH from $\lambda_f(p)/2=\cos(\theta_f(p))$ to $\X(p)=\big(-\log (1-e^{i\theta_f(p)}/p)-\log (1-e^{-i\theta_f(p)}/p)\big)/(-2 \log(1-1/p))$.
\end{proof}

\subsection{Kloosterman sums} 
Let $q$ be a large prime. The classical (normalized) Kloosterman sums modulo $q$ are defined by 
$$
\textup{Kl}_q(a,b):=\frac{1}{\sqrt{q}}\sum_{n\in \Fs} e_q(an+b\overline{n}),$$
where $e_q(z)=\exp(2\pi i z/q)$ and $\overline{n}$ is the multiplicative inverse of $n$ modulo $q$.

The analogue of Lemma 2.1 of \cite{La1} for Kloosterman sums (which follows in a similar way from Proposition 4.2 of \cite{Pe}), see also Proposition 9.9 of \cite{AuBoLa} for a more general result, implies that for all positive integers $1\leq \ell \leq (\log q)/2$, and all $n_1, \dots, n_{\ell}\in \mathbb{F}_q$ we have 
\begin{equation}\label{MomentKloosterman}
\frac{1}{(q-1)^2}\sum_{(a, b)\in (\Fs)^2} \textup{Kl}_q(a-n_1, b) \cdots \textup{Kl}_q(a-n_{\ell}, b) = \ex\big(\Y(n_1)\cdots  \Y(n_k)\big) + O\left(\frac{2^{\ell} \ell}{\sqrt{q}}\right),
\end{equation}
where for each $n$ we have $\Y(n)=2\cos(\theta_n)$, where the $\theta_n$ are independent random variables with Sato-Tate distribution on $[0, \pi]$. 
This shows that the sequence $\X(n)=\textup{Kl}_q(a-n, b)/2$ satisfies the AIH with $Q=q$ and $L \sim (\log q)/(2\log 2+ 2).$

The asymptotic \eqref{MomentKloosterman} is proved using deep ingredients from algebraic geometry, including Deligne's equidistribution theorem which generalizes the Riemann hypothesis over finite fields, and Katz's work on monodromy groups. 
This estimate was a key ingredient in investigating the distribution of partial sums of Kloosterman sums in \cite{AuBoLa} and \cite{KoSa}.

 \section{Estimates for the sum $H(s)$}

 In this section, we obtain estimates for $H(s)$ which will be useful in our subsequent work. We start with the following lemma which easily follows from partial summation and the growth rate of $\mathcal A(x)$.
 
\begin{lemma}\label{LogSum}
  If  $s\geq 4\beta'$ then
\[
H(s)= \begin{cases}
\frac{\alpha}{A}(\log s)^{A} +C_{\mathcal A} - \alpha(\log \beta) (\log s)^{A-1} +O\left((\log s)^{ \max\{A-2, A-B\}}\right)   & \text{ if } A>0,\\
       \alpha \log\log s+C_{\mathcal A}  - \alpha(\log \beta) (\log s)^{-1} +O\left((\log s)^{ \max\{-2, -B\}}\right) & \text{ if } A=0.
 \end{cases}
 \]
\end{lemma}
\begin{proof}
Let $E(x)=\mathcal A(x)- \alpha \int_2^x (\log (t/\beta))^{A-1}  dt$. By partial summation we have
\begin{equation}\label{SumRecip}
H(s) = \sum_{q_n\leq 2} \frac 1{q_n}+ \int_2^s \frac{d\mathcal A(t)}t = \alpha \int_2^s  (\log (t/\beta))^{A-1}  \frac{dt}t + \int_2^s \frac{dE(t)}t+\sum_{1\leq q_n\leq 2} \frac 1{q_n}.
\end{equation}
Since $E(t) \ll t(\log t)^{A-1-B}$ and $B>A$, the integral  $\int_2^s dE(t)/t$ is convergent, and hence we get 
$$ H(s) =\begin{cases} \frac{\alpha}{A}(\log s/\beta)^{A} +C_{\mathcal A}  + \E_1(s) & \text{ if } A>0,\\
 \alpha\log\log s/\beta+ C_{\mathcal A} + \E_1(s)& \text{ if } A=0,
\end{cases}
$$
where 
$$ \E_1(s) \ll  \int_s^{\infty}   \frac{dt}{t(\log t)^{B-A+1}}\ll (\log s)^{ A-B}.
$$ 
The result follows from expanding the $\log s/\beta$ term on the right-hand side in each case.
\end{proof}

Our next result gives estimates for the difference $H(u)-H(v)$, which are needed to prove Theorems \ref{cor: Main2*} and \ref{cor: Main2}.

\begin{lem} \label{DiffL}Let $u, v\geq 2$, and put $\tau=\log u -\log v$. If $|\tau|\leq (\log v)/2$ then 
$$
H(u)-H(v)= \alpha \tau (\log v)^{A-1} +O\left( (\tau+\tau^2)(\log v)^{A-2}+ (\log v)^{ A-B }\right).$$
\end{lem}
\begin{proof} By definition we have 
\[
H(u)-H(v) = \int_v^u \frac{d\mathcal A(t)}{t} = \alpha\int_v^u \frac{(\log t/\beta)^{A-1}} t dt +O( (\log v)^{A-B}).
\]
Let $L=\log v/\beta$.
If $A>0$ then the main term here is 
\[
\frac \alpha{A} ((L+\tau)^{A}-L^{A}) = \alpha \tau L^{A-1} +O(\tau^2 L^{A-2})
= \alpha \tau (\log v)^{A-1}+O((\tau+\tau^2) L^{A-2}).
\]
If $A=0$ then the main term is 
\[
\alpha(\log (L+\tau)-\log L)=\frac {\alpha \tau} L +O\bigg(\frac{\tau^2}{L^2}\bigg)
=\frac {\alpha \tau} {\log v} +O\bigg(\frac{\tau+\tau^2}{L^2}\bigg) .\qedhere
\] 
\end{proof}


\section{Estimates for the cumulant generating function of $H_{\Y}(Q)$}

Throughout this section we assume that the independent random variables $\{\Yn\}_{q_n\leq Q}$   satisfy CRVH$(\Y)$.
For every complex number $s$, we denote the cumulant generating function for  the sum  $H_{\Y}(Q)$, by 
\[
K_Q(s)= K_Q(\Y, s):= \log \ex\left(e^{sH_{\Y}(Q)}\right)= \log \mathbb \ex\left(e^{s\sum_{q_n\leq Q}  \Yn/q_n}\right).
\]
By the independence of the $\Yn$ we have
\begin{equation}\label{Cumu}
K_Q(s)=\sum_{q_n\leq Q} \log \ex\left(e^{s\Yn/q_n}\right)= s H(s) + \sum_{q_n\leq Q}    \fn\left(\frac{s}{q_n}\right),
\end{equation}
extending the definition of the $\fn$   in \eqref{Thefunctionf}  to  complex $s$,  the two cases corresponding to $|s|<1$ and $|s|\geq 1$.
In this section we will show that $\lim_{Q\to\infty} K_Q(s)$ exists for all $s\in \mathbb{C}$, and we will estimate $K_Q(s)$ in a large range of $s$ in terms of $Q$. These will be important ingredients  in the proofs of Theorems \ref{cor: Main2*} and \ref{LimitLaw}.  

We will need to prove the following estimates for the function $\fn$. 

\begin{lemma}\label{LogEx}
For each $n$ we have uniformly
\begin{equation}\label{LogEx1}
\fn(t)\ll \begin{cases} t^2 & \text{ if }  0\leq t<1,\\ \displaystyle{\frac{t}{\log^2(2t)}}  & \text{ if } t\geq 1,\end{cases}
\end{equation}
(which implies that the integral defining $\eta_{\Y}:=\int_0^{\infty} \frac{f_{\Y}(u)}{u^2}du$ converges) and 
\begin{equation}\label{LogEx2}
\fn'(t)\ll \begin{cases} t & \text{ if } 0<t<1,\\ \displaystyle{\frac{1}{\log^2(2t)}} & \text{ if } t> 1.\end{cases}
\end{equation}
\end{lemma}

\begin{proof}[Proof of Lemma \ref{LogEx}]
We start by proving \eqref{LogEx1}. Let $s$ be a complex number such that $|s|< 1$. Using the Taylor expansion of the exponential we have 
\[
\ex(e^{s\Yn})= \ex(1+s\Yn+ O(|s|^2\Yn^2))= 1+ O(|s|^2)
\]
since $\ex(\Yn)= 0$ and $|\Yn|\leq 1$. This implies
\begin{equation}\label{LogEx0}
\log \ex(e^{s\Yn}) \ll |s|^2,
\end{equation} 
which gives the desired estimate for $\fn(t)$ when $0\leq t< 1$.

We now suppose that $t\geq 1$. Let $N>1$ be a parameter to be chosen. Then we have
$$ \ex(e^{t\Yn}) \geq \pr(\Yn>1-1/N) e^{t(1-1/N)} \gg \exp \left(t(1-1/N)-ce^{\sqrt{N}}\right),$$
by \eqref{LowerBoundf}.
Choosing $N=(\tfrac 12\log(2t/c^2))^2$ and using that $\Yn\leq 1$ we obtain
\begin{equation}\label{MomentGenerating}
\exp \left(t-\frac{5t}{\log^2 (2t/c^2)}\right) \ll \ex(e^{t\Yn}) \leq e^{t},
\end{equation}
from which the desired estimate for $\fn(t)$ follows in this case. 

Next, we establish \eqref{LogEx2}. Note that $\fn$ is differentiable on $(0, \infty)\setminus\{1\}$ and we have 
\begin{equation}\label{Derivf}
\fn'(t)=\begin{cases} \displaystyle{\frac{\ex (\Yn e^{t\Yn})}{\ex (e^{t\Yn})}}& \text{ if } 0<t<1,\\ \displaystyle{\frac{\ex (\Yn e^{t\Yn})}{\ex (e^{t\Yn})}}- 1 & \text{ if } t> 1.\end{cases}
\end{equation}
As before, in the case $0<t<1$ the estimate of $\fn'(t)$ follows from the Taylor expansions $\ex(e^{t\Yn})=1+O(t^2)$ and 
\[
\ex(\Yn e^{t\Yn})= \ex(\Yn+t\Yn^2+ O(t^2|\Yn^3|))= t\ex(\Yn^2)+ O(t^2).
\] 

We now suppose that $t>1$, and let $N>1$ be a parameter to be chosen. Let $\mathcal{B}$ be the event $\Yn>1-1/N$, and $ \mathcal{B}^{c}$ be its complement. Then we have 
$$ \ex (\Yn e^{t\Yn})= \ex(\mathbf{1}_{\mathcal{B}} \cdot \Yn e^{t\Yn}) + \ex(\mathbf{1}_{\mathcal{B}^c} \cdot \Yn e^{t\Yn})\geq (1-1/N) \ex(\mathbf{1}_{\mathcal{B}} \cdot  e^{t\Yn})+ O(e^{t(1-1/N)}) ,$$
 since $\ex(\mathbf{1}_{\mathcal{B}^c} \cdot e^{t\Yn})\leq e^{t(1-1/N)}$ . Moreover this then implies that
\begin{equation}\label{MomentGen2} \ex (\Yn e^{t\Yn}) \geq (1-1/N) \ex( e^{t\Yn})+ O(e^{t(1-1/N)}).
\end{equation}
We choose $N= (\log 2t/c^2)^2/6$. Then, it follows from \eqref{MomentGenerating} that 
$$ e^{t(1-1/N)} \ll \frac{\ex(e^{t\Yn})}{t}.$$
Inserting this estimate in \eqref{MomentGen2}, and using the bound $\Yn\leq 1$ gives 
$$ 1- \frac{1}{N}+O\left(\frac{1}{t}\right) \leq \frac{\ex (\Yn e^{t\Yn})}{\ex (e^{t\Yn})} \leq 1.$$
This implies the desired estimate for $\fn'$ in this case, as desired. 
\end{proof}

\begin{corollary}\label{Cor.cumulant}
For each complex number $s$ the limit $\lim_{Q\to \infty} K_Q(s)$ exists, and equals $K(s)$, say. Moreover $ K_Q(s)=K(s)+O( |s|^2  \tfrac{(\log Q)^{A-1}}Q)$ if $Q>|s|$.
\end{corollary}

\begin{proof}
If $|s|<Q<q_n\leq Z$ then $|s|/q_n< 1$ and so 
\[
\log \ex\left(e^{s\Yn/q_n}\right)  \ll   \left(\frac{|s|}{q_n}\right)^2.
\]
by \eqref{LogEx0}. Therefore,  by \eqref{Cumu},
\begin{align*} 
K_Z(s)-K_Q(s) = \sum_{Q<q_n\leq Z}    \log \ex\left(e^{s\Yn/q_n}\right)
\ll \sum_{Q<q_n\leq Z}    \left(\frac{|s|}{q_n}\right)^2   \ll  |s|^2\frac{ (\log Q)^{A-1}}Q
\end{align*}
  by the growth rate of $\mathcal A(x)$. This is $o_{Q\to \infty}(1)$ and so the 
$K_Q(s)$ form a Cauchy sequence, and therefore have a limit. This same calculation then implies the claimed estimate.
\end{proof}


Next we wish to estimate the cumulant generating function $K_Q(s)$ in a large range of $s$ in terms of $Q$. 
\begin{proposition}\label{Cumulant}
For any real number $s$ in the range $3\leq s\leq  Q/\log Q$ we have
$$K_Q(s)= sH(s)+ \left(\alpha \eta_{\Y} +O\left(\frac{\log\log s}{\log s}\right)\right)\,   s  (\log s)^{A-1} .$$
\end{proposition}

\begin{proof}[Proof of Proposition \ref{Cumulant}]
By \eqref{Cumu} it suffices to show that for $3\leq s\leq Q/\log Q$ we have
$$ \sum_{q_n\leq Q}    \fn\left(\frac{s}{q_n}\right)= \left(\alpha \eta_{\Y} +O\left(\frac{\log\log s}{\log s}\right)\right)\,   s  (\log s)^{A-1}. 
$$
To estimate this sum, we shall split it in three parts: $q_n\leq s^{1/2}$, $s^{1/2}<q_n<s\log s$ and $s\log s \leq q_n\leq Q$. By \eqref{LogEx1} and Lemma \ref{LogSum} the contribution of the first part is 
$$ \sum_{q_n\leq s^{1/2}} \fn\left(\frac{s}{q_n}\right)\ll \frac{s}{\log^2 s}\sum_{q_n\leq s^{1/2}} \frac{1}{q_n}\ll s (\log s)^{A-2}\log\log s.$$

Moreover, using \eqref{LogEx1}, we deduce that the contribution of the last part is 
\begin{align*}
\sum_{s\log s \leq q_n\leq Q} \fn\left(\frac{s}{q_n}\right) &\ll s^2\sum_{s\log s \leq q_n\leq Q}\frac{1}{q_n^2}= s^2\int_{s\log s}^{Q} \frac{d\mathcal A(t)}{t^2} \ll s(\log s)^{A-2},
\end{align*}
by the growth rate of $\mathcal{A}(x)$. Hence, we deduce that 
\begin{equation}\label{TruncSum}
\sum_{q_n\leq Q} \fn\left(\frac{s}{q_n}\right)=\sum_{s^{1/2}<q_n<s \log s} \fn\left(\frac{s}{q_n}\right) +O(s (\log s)^{A-2}\log\log s).
\end{equation}

By the third hypothesis of  CRVH$(\Y)$ we see that for all $t$ 
\[
\ex (e^{t\Yn}) =\ex (e^{t\Y_n+O(\tfrac t{(\log n)^{2}})} )=\ex (e^{t\Y}) e^{O(\tfrac t{(\log n)^{2}})}
\]
so that $f_{\Yn}(t)=\f(t)+O(\tfrac t{(\log n)^{2}})$. Therefore
\[
\sum_{s^{1/2}<q_n<s \log s} \fn\left(\frac{s}{q_n}\right) = \sum_{s^{1/2}<q_n<s \log s} f_{\Y}\left(\frac{s}{q_n}\right)
+O\bigg(  \sum_{s^{1/2}<q_n<s \log s}  \frac s{q_n(\log n)^{2}} \bigg) 
\]
For this last sum we note that if $s^{1/2}<q_n<s \log s$ then $(\log n)^{2} \asymp (\log s)^{2}$   and so
\[
 \sum_{s^{1/2}<q_n<s \log s}  \frac s{q_n(\log n)^2} \ll \frac s{(\log s)^2} H(s \log s)\ll   s(\log s)^{A-2} \log\log s
\]
by Lemma \ref{LogSum} and so
\begin{equation}\label{TruncSum2}
\sum_{q_n\leq Q} \fn\left(\frac{s}{q_n}\right)=\sum_{s^{1/2}<q_n<s \log s}  f_{\Y}\left(\frac{s}{q_n}\right) +O(s (\log s)^{A-2}\log\log s).
\end{equation}

Now, we have, for $E(x)=\mathcal A(x)- \alpha \int_2^x (\log (t/\beta))^{A-1}  dt\ll x (\log x)^{A-1-B}$,
\begin{equation}\label{MPart}
\sum_{s^{1/2}<q_n<s \log s} f_{\Y}\left(\frac{s}{q_n}\right)=\int_{s^{1/2}}^{s\log s} f_{\Y}(s/t)d\mathcal A(t)= \alpha \int_{s^{1/2}}^{s\log s} f_{\Y}(s/t) \log^{A-1}(t/\beta)dt +\E_1,
\end{equation}
where 
\begin{align*}
 \E_1& = \int_{s^{1/2}}^{s\log s} f_{\Y}(s/t) dE(t) =[E(t)f_{\Y}(s/t) ]_{s^{1/2}}^{s\log s} + s\int_{s^{1/2}}^{s\log s} f_{\Y}'(s/t) E(t) \frac{dt}{t^2} \\
 &\ll  f_{\Y}((\log s)^{-1}) s(\log s)^{A-B} + f_{\Y}(s^{1/2}) s^{1/2} (\log s)^{A-1-B}+ \int_{s^{1/2}}^{s\log s} \frac{s}{t^2}|f_{\Y}'(s/t)| t (\log t)^{A-1-B}dt\\
 & \ll s (\log s )^{A-B-2}+  s (\log s)^{A-1-B} \int_{s^{1/2}}^{s\log s}\frac{|f_{\Y}'(s/t)|}{t} dt. 
 \end{align*}
by \eqref{LogEx1}. To bound this last integral, we use \eqref{LogEx2}. This gives
$$ \int_{s^{1/2}}^{s\log s}\frac{|f_{\Y}'(s/t)|}{t} dt \ll \int_{s^{1/2}}^s \frac{1}{t\log^2(2s/t)}dt+ \int_{s}^{s\log s} \frac{s}{t^2} dt \ll \int_2^{2s^{1/2}} \frac{1}{v\log^2 v} dv+ 1\ll 1,$$
by making the change of variable $v=2s/t$.
Inserting these estimates in \eqref{MPart} gives 
\begin{equation}\label{MPart2}
\sum_{s^{1/2}<q_n<s\log s} f_{\Y}\left(\frac{s}{q_n}\right)= \alpha \int_{s^{1/2}}^{s\log s} f_{\Y}(s/t) \log^{A-1}(t/\beta)dt+ O\left(s (\log s)^{A-1-B}\right).
\end{equation}
Using the change of variables $u=s/t$ we obtain
\begin{align*}
\int_{s^{1/2}}^{s\log s} f_{\Y}(s/t) \log^{A-1}(t/\beta)dt 
&= s\int_{(\log  s)^{-1}}^{s^{1/2}} \frac{f_{\Y}(u)}{u^2} \log^{A-1}(s/u\beta)du\\
&= s\log^{A-1}(s/\beta)\int_{(\log  s)^{-1}}^{s^{1/2}} \frac{f_{\Y}(u)}{u^2} \left(1+O\left(\frac{\log u}{\log s}\right)\right)du.
\end{align*}
Moreover, using \eqref{LogEx1} we get
$$ \int_{0}^{(\log  s)^{-1}} \frac{f_{\Y}(u)}{u^2}du+ \int_{s^{1/2}}^{\infty} \frac{f_{\Y}(u)}{u^2}du\ll \int_{0}^{(\log  s)^{-1}} 1du+ \int_{s^{1/2}}^{\infty} \frac{1}{u(\log u)^2}du \ll \frac{1}{\log s}, 
$$
and 
$$ 
\int_{(\log  s)^{-1}}^{s^{1/2}} \frac{f_{\Y}(u)\log u}{u^2}du \ll \int_{(\log s)^{-1}}^{1} |\log u|du+\int_{1}^{s^{1/2}} \frac{1}{u\log u}du \ll \log\log s.
$$ 
This gives 
$$ \int_{s^{1/2}}^{s\log s} f_{\Y}(s/t) \log^{A-1}(t/\beta)dt =\eta_{\Y} s(\log s)^{A-1} \left(1+O\left(\frac{\log\log s}{\log s}\right)\right). 
$$
Combining this estimate with \eqref{TruncSum2} and \eqref{MPart2} completes the proof.
\end{proof}

\section{The main Theorem for the sum of independent random variables : Proof of Theorem \ref{cor: Main2*}}

The key technical result in this paper is  the following result:

\begin{thm}\label{Main*}
Let $Q$ be large and $\{\Yn\}_{q_n\leq Q}$ be a sequence of independent random variables which satisfy CRVH$(\Y)$. 
Given $V$ in the range 
\begin{equation}\label{DefRLQ1}
c \leq V\leq R:= \begin{cases} \frac{\alpha}{A} (\log Q)^{A} +C_{\mathcal A}- \big( \alpha  \log\log Q +\Theta\big) (\log Q)^{A-1}  & \text { if } A>0, \\
  \alpha \log\log Q +C_{\mathcal A} -(\alpha \log \log Q +\Theta)/\log Q & \text { if } A=0, \\
 \end{cases}
 \end{equation}
where $\Theta>0$ and $c>1$ are suitably large constants, select $Z=Z(V)$ to be the smallest number $z$ for which $H(z)\geq  V$.
Then we have 
$$ \Phi_Q(\Y, V)= \exp\left(-\alpha e^{-\eta_{\Y}-1} Z (\log Z)^{A-1} \left(1+O\left(\sqrt{\E(Z)}\right)\right)\right).$$
where $\E(Z)$ is defined by\footnote{The constant $c$ above is chosen to be large enough that $Z$ is large enough so that $\E(Z)$ is well-defined.}
 \begin{equation}\label{EpsilonS}
\E(Z):= \frac{\log\log Z}{\log Z}+\frac{1}{(\log Z)^{B-1}}.
\end{equation}
\end{thm}

\begin{proof}[Proof of Theorem \ref{Main*}]
We shall use the saddle-point method to prove Theorem \ref{Main*}. For $V\geq 1$, we define  
\begin{equation}\label{Def s}
s(V):=e^{-\eta_{\Y}-1} Z(V).
\end{equation}
 The definition of $Z$ implies that $H(Z)+O(1/Z)=V\leq R$. Therefore by inserting the estimate for $H(Z)$ in Lemma \ref{LogSum} into
 \eqref{DefRLQ1} we deduce that 
 \[
 s(V)\ll \frac Q{\log Q},
 \]
 and so for any $s\asymp s(V)$ we have
\begin{equation}\label{EstimateKQS}
K_Q(s)= sH(s)+\Big(\alpha \eta_{\Y} +O\big(\E(s(V))\big)\Big)\,   s  \mathcal L^{A-1} 
\end{equation}
by  Proposition \ref{Cumulant}, where for the rest of this proof we write $\mathcal L:=\log s(V)$.

For $u>0$ we have 
\begin{equation}\label{IntDistrib}
\exp(K_Q(u))= -\int_{-\infty}^{\infty} e^{ut} d\Phi_Q(t)= u\int_{-\infty}^{\infty} e^{ut} \Phi_Q(t)dt.
\end{equation}
We will now show (in \eqref{BulkDistrib} below) that for $u=s$ by far the largest part of the integral on the right-hand side occurs for $t$ in a short interval around $s$:

Let $0<\delta<1$ be a small parameter to be chosen, and put $S_1=e^{\delta}s$, writing $s=s(V)$. Let $\Delta=\delta\alpha \mathcal L^{A-1}$. Then, it follows from \eqref{EstimateKQS}  and \eqref{IntDistrib} that 
\begin{equation}\label{CutInt}
\begin{aligned}
\int_{V+\Delta} ^{\infty} e^{st} \Phi_Q(t)dt &\leq \exp((s-S_1)(V+\Delta)) \int_{V+\Delta}^{\infty} e^{S_1t} \Phi_Q(t)dt\\
& \leq \exp\Big((s-S_1)(V+\Delta)+ S_1 H(S_1) + (\alpha \eta_{\Y} +O(\E(s))) S_1 \mathcal L^{A-1} \Big)  .\\
\end{aligned}
\end{equation}
Now, by Lemma \ref{DiffL} we have 
$$ 
S_1 H(S_1) + \alpha\eta_{\Y} S_1 \mathcal L^{A-1} = e^{\delta}\left(sH(s)+ \alpha\eta_{\Y} s \mathcal L^{A-1}\right)+ \alpha e^{\delta}  \delta s  \mathcal L^{A-1} + O(s \mathcal L^{A-1}\E(s) ).
$$
Furthermore, by Lemma \ref{DiffL} we have
\begin{equation}\label{estimateV}
V= H(Z)+O\left(\frac{1}{Z}\right)= H(s)+ (\alpha (\eta_{\Y}+1)+O(\E(s))) \mathcal L^{A-1} .
\end{equation}
Inserting these estimates in \eqref{CutInt}  we get
$$ \int_{V+\Delta} ^{\infty} e^{st} \Phi_Q(t)dt \leq \exp\Big(sH(s) + \alpha \eta_{\Y} s\mathcal L^{A-1}+ (1+\delta-e^{\delta})\alpha s\mathcal L^{A-1} +O(s \mathcal L^{A-1}\E(s)) \Big).$$
Therefore, choosing $\delta=B_0 \sqrt{\E(s)}$, for a suitably large constant $B_0$, and using \eqref{EstimateKQS}, we deduce that 
\begin{equation}\label{TailInt1}
\int_{V+\Delta} ^{\infty} e^{st} \Phi_Q(t)dt \leq e^{-s \mathcal L^{A-1}\E(s)} \int_{-\infty} ^{\infty} e^{st} \Phi_Q(t)dt.
\end{equation}
A similar calculation shows that by putting $S_2=e^{-\delta}s$, we obtain
\begin{align*}
\int_{-\infty} ^{V-\Delta} e^{st} \Phi_Q(t)dt &\leq \exp((s-S_2)(V-\Delta)) \int_{-\infty}^{V-\Delta} e^{S_2t} \Phi_Q(t)dt\\
&\leq  \exp\Big((s-S_2)(V-\Delta)+ S_2 H(S_2) + (\alpha\eta_{\Y} +O(\E(s))) S_2 \mathcal L^{A-1} \Big) \\
&=\exp\Big( sH(s) + \alpha \eta_{\Y} s\mathcal L^{A-1} + (1-\delta-e^{-\delta})\alpha s\mathcal L^{A-1} +O(s \mathcal L^{A-1}\E(s))\Big),
\end{align*}
and hence by our choice of $\delta$ and \eqref{EstimateKQS} we deduce
\begin{equation}\label{TailInt2}
\int_{-\infty} ^{V-\Delta} e^{st} \Phi_Q(t)dt \leq e^{-s \mathcal L^{A-1}\E(s)} \int_{-\infty} ^{\infty} e^{st} \Phi_Q(t)dt.
\end{equation}
Combining the bounds \eqref{TailInt1} and \eqref{TailInt2} with \eqref{EstimateKQS} and \eqref{IntDistrib}  gives 
\begin{equation}\label{BulkDistrib}
\int_{V-\Delta} ^{V+\Delta} e^{st} \Phi_Q(t)dt= \exp\Big(sH(s) + \alpha \eta_{\Y} s\mathcal L^{A-1}  \left(1+O(\E(s))\right)\Big),
\end{equation}
as claimed.

 Now $\Phi_Q(t)$ is non-increasing as a function of $t$, by definition, and so we may bound the above integral as follows 
 $$ e^{sV+O(s\Delta)} \Phi_Q(V+\Delta)\leq \int_{V-\Delta} ^{V+\Delta} e^{st} \Phi_Q(t)dt \leq e^{sV+O(s\Delta)} \Phi_Q(V-\Delta).$$
Inserting these bounds in \eqref{BulkDistrib} and using \eqref{estimateV} we obtain
\begin{equation}\label{BDelta}
\Phi_Q(V+\Delta) \leq  \exp\left(-\alpha s \mathcal L^{A-1}\left(1+ O\left(\sqrt{\E(s)}\right)\right)\right)  \leq \Phi_Q(V-\Delta).
\end{equation}
Let $s^+$ and $s^-$ be defined by $s^{\pm}=  e^{-\eta_{\Y}-1} Z(V\pm \Delta)$. Then, it follows from two applications of Lemma \ref{DiffL} that
\begin{align*}
V\pm\Delta& =H(Z(V\pm \Delta)) +O\left(1\right) \\
& = H(s^{\pm}) + \alpha (\eta_{\Y}+1) (\log s^{\pm})^{A-1} + O\left((\log s^{\pm})^{ \max\{A-2, A-B\}}\right) \\
& = H(s) +\alpha t^{\pm} \mathcal L^{A-1}+ \alpha (\eta_{\Y}+1) (\log s +t^{\pm})^{A-1} + O\left((\log s^{\pm})^{ \max\{A-2, A-B\}}\right)
\end{align*}
where $t^{\pm}= \log s^{\pm}-\log s.$ By \eqref{estimateV}  we deduce that
$$
 \pm \Delta=  \alpha t^{\pm} \mathcal L^{A-1} + \alpha(\eta_{\Y}+1) \big((\log s +t^{\pm} )^{A-1}- \mathcal L^{A-1}\big) + O\left(\mathcal{E}(s^{+})(\log s^{+})^{A-1}\right).
 $$
Since $\Delta=o(\mathcal L^{A-1})$ this implies that $t^{\pm}=o(1)$, and so 
$$ \alpha t^{\pm} \mathcal L^{A-1}=\pm \Delta + O(\mathcal{E}(s)\mathcal L^{A-1}).$$
Since $\Delta=\delta\alpha \mathcal L^{A-1}$, this implies that 
$  t^{\pm} =\pm\delta  + O(\E(s)) \ll \sqrt{\E(s)}$, and thus
\begin{equation}\label{SPlus}
s^{\pm} (\log s^{\pm})^{A-1}= s \mathcal L^{A-1}\left(1+ O\left(\sqrt{\E(s)}\right)\right).
\end{equation}
Finally, using the upper bound of \eqref{BDelta} with $V+\Delta$ instead of $V$ together with this last estimate gives
$$ \Phi_Q(V) \geq \exp\left(-\alpha s^+ (\log s^+)^{A-1}\left(1+ O\left(\sqrt{\E(s^+)}\right)\right)\right)= \exp\left(-\alpha s\mathcal L^{A-1}\left(1+ O\left(\sqrt{\E(s)}\right)\right)\right).$$
A similar lower bound follows from the lower of \eqref{BDelta} with $V-\Delta$ instead of $V$, together with \eqref{SPlus}.
\end{proof}


\begin{proof}  [Deduction of Theorem \ref{cor: Main2*}]
To deduce Theorem \ref{cor: Main2*} from Theorem \ref{Main*}, we  express $Z=Z(V)$ as an explicit function in terms of $V$, using 
the definition of $Z$ and Lemma \ref{LogSum}:
If $A>0$ then we have
$$
V-C_{\mathcal A}= \frac{\alpha}{A}(\log Z)^{A} \left(1- \frac{A\log \beta}{\log Z}+ O\left((\log Z)^{ \max\{-2, -B\}}\right)\right).
$$
Re-arranging this gives
\begin{align*}
W&= \log Z \left(1- \frac{A\log \beta}{\log Z}+O\left((\log Z)^{ \max\{-2, -B\}}\right)\right)^{1/A}\\
&= \log Z -\log \beta+O\left((\log Z)^{ \max\{-1, 1-B\}}\right),
\end{align*}
and hence, by exponentiating, we obtain 
\begin{equation}\label{sCase1}
Z= \beta e^W\left(1+O\left(W^{\max\{-1 ,  1-B \}}\right)\right).
\end{equation}
If $A=0$ then, analogously,
$$ V-C_{\mathcal A}=\alpha \log \log  Z - \frac{\alpha\log \beta}{\log Z} + O\left(\frac{1}{(\log Z)^{ \min\{2, B\}}}\right)$$
so that 
\begin{align*}
W
&= (\log Z) \exp\left(-\frac{\log \beta}{\log Z}+ O\left(\frac{1}{(\log Z)^{ \min\{2, B\}}}\right)\right)\\
&= \log Z -\log \beta+O\left(\frac{1}{(\log Z)^{ \min\{1, B-1\}}}\right),
\end{align*}
and hence \eqref{sCase1} holds by exponentiating.

The estimate claimed for $ \Phi_Q(V)$ now follows by substituting \eqref{sCase1}  into the estimate  for $ \Phi_Q(V)$ in Theorem \ref{Main*}.

We also deduce the range for $W$ from the range for $V$ in Theorem \ref{Main*} by an argument similar to that above, with any $\Theta'> \Theta/\alpha$.
\end{proof}

 
\section{Using the AIH Hypothesis: Proofs of Theorems \ref{cor: Main2} and \ref{LimitLaw}}

We first start by using the Approximate Independence Hypothesis to relate the moment generating function of $H_{\X}(s)$ to that of $H_{\Y}(s)=\sum_{q_n\leq Q}  \Yn/q_n$, where $\Y(n)$ are independent random variables. 

\begin{pro}\label{LaplaceTransformQ} Assume AIH$(L,Q)$. For all complex numbers $s$ such that $|s|\leq L/(10 H(Q))$, we have 
$$\ex \left(e^{sH_{\X}(Q)}\right)= \ex \left(e^{sH_{\Y}(Q)}\right) +O\left(e^{-L/2}\right).$$
\end{pro}

\begin{proof}
First note that 
\begin{equation}\label{FirstApprox}
\ex \left(e^{sH_{\X}(Q)}- e^{sH_{\Y}(Q)}\right)= \sum_{k\leq L} \frac{s^k}{k!} \
\ex\left( H_{\X}(Q)^k -H_{\Y}(Q)^k   \right)+ \text{Err},
\end{equation}
where, since $|H_{\X}(Q)|, |H_{\Y}(Q)|\leq H(Q)$, we have
$$ \text{Err} \leq 2\sum_{k> L} \frac{|s|^kH(Q)^k}{k!} \ll \sum_{k>L}\left(\frac{3|s| H(Q)}{L}\right)^k\ll e^{-L},
$$
by Stirling's formula and our assumption on $s$. We now expand the moments on the right hand side of \eqref{FirstApprox}, and use the Approximate Independence Hypothesis to get
\[
\ex\left( H_{\X}(Q)^k -H_{\Y}(Q)^k   \right)
=\sum_{q_{n_1}, \dots, q_{n_k}\leq Q} \frac{\ex\big(\X(n_1)\cdots \X(n_k)-\Y(n_1)\cdots \Y(n_k)\big)}{q_{n_1}\cdots q_{n_k}} 
  \ll H(Q)^k e^{-L}.
\]
Inserting this estimate into \eqref{FirstApprox} allows us to bound the sum on the right-hand side by 
$$ 
\ll e^{-L}. \sum_{k\leq L} \frac{(|s|H(Q))^k}{k!} \leq e^{|s| H(Q)-L}\leq e^{-L/2},
$$
by our assumption on $s$, and the result follows.
\end{proof}

Combining this result with Corollary \ref{Cor.cumulant}, we prove that the characteristic function of $H_{\X}(Q)$ converges, as $Q\to\infty$, to the characteristic function of some random variable.

\begin{pro}\label{CharLimit} Assume AIH$(L,Q)$ and CRVH$(\Y)$.
There exists a function $\varphi(t)$ such that $\ex(e^{itH_{\X}(Q)})$ converges pointwise to $\varphi(t)$ on $\mathbb{R}$. Moreover, for all real numbers $|t|\leq \min\Big( L/(10H(Q)), Q^{1/2}(\log Q)^{-\frac{A-1}{2}}\Big)$ we have 
\begin{equation}\label{ConvChar}
\ex(e^{itH_{\X}(Q)})= \varphi(t) + O\left(e^{-L/2}+   \frac{t^2(\log Q)^{A-1}}{Q}\right)
\end{equation} 
(so that $\ex(e^{itH_{\X}(Q)})$ converges uniformly to $\varphi(t)$ on any finite interval).
Furthermore, the function $\varphi$ satisfies
\begin{equation}\label{DecayChar}
 \varphi(t) \ll \exp\left(-C_0 |t|(\log |t|)^{A-2}\right), 
 \end{equation}
for all real numbers $|t|\geq 2$ where $C_0$ is an absolute constant. 
\end{pro}

\begin{proof}
Let $\varphi(t)=e^{K(it)}$, where $K(s)=\lim_{Q\to \infty} K_Q(\Y,s)$ from Corollary \ref{Cor.cumulant}. Then, it follows from this result that for $|t|\leq Q^{1/2}(\log Q)^{-\frac{A-1}{2}}$ we have
\begin{equation}\label{ApproxChar}
\begin{aligned}
 \varphi(t)&= \exp\left(K_Q(\Y,it)+O\left(\frac{t^2(\log Q)^{A-1}}{Q}\right)\right)\\
 & = \ex(e^{itH_{\Y}(Q)}) \left(1+O\left(\frac{t^2(\log Q)^{A-1}}{Q}\right)\right)\\
 &= \ex(e^{itH_{\Y}(Q)}) + O\left(\frac{t^2(\log Q)^{A-1}}{Q}\right).
\end{aligned} 
\end{equation}
Combining this estimate with Proposition \ref{LaplaceTransformQ} gives \eqref{ConvChar}. 

Hence, it remains to prove the bound \eqref{DecayChar}. We will prove that this bound is verified by $\ex(e^{itH_{\Y}(Q)})$, and then use \eqref{ApproxChar} to deduce the same bound for $\varphi(t)$. Let $|t|\geq 2$ be a real number, and assume that $Q$ is large so that $Q> t^2$. Since the $\Y(n)$ are independent we have 
$$ 
\left|\ex(e^{itH_{\Y}(Q)})\right|=\bigg|\prod_{q_n\leq Q}\ex\left(e^{it\Y(n)/q_n}\right)\bigg| \leq \prod_{n:\ |t|\log |t|<q_n<t^2}\bigg|\ex\left(e^{it\Y(n)/q_n}\right)\bigg|.
$$
Now, since $q_n>|t|$, we can use the Taylor expansion of the exponential to get 
\begin{align*}
\ex\left(e^{it\Y(n)/q_n}\right)
&= \ex\left(1+ \frac{it}{q_n} \Y(n)- \frac{t^2}{2q_n^2}\Y(n)^2 + O\left(\frac{t^3}{q_n^3}\right)\right)\\
&= 1- \frac{t^2}{2q_n^2}\ex(\Y(n)^2)+O\left(\frac{t^3}{q_n^3}\right) \\
& = \exp\bigg( - \frac{t^2}{2q_n^2}\ex(\Y(n)^2)+O\left(\frac{t^3}{q_n^3}\right) \bigg) 
\end{align*}
as each $|\Y(n)|\leq 1$.
Furthermore, it follows from \eqref{LowerBoundf} that there exists an absolute constant $c_1>0$ such that $\ex(\Y(n)^2)\geq c_1$ for all $n\geq 1$. Therefore, we deduce that 
$$ \left|\ex(e^{itH_{\Y}(Q)})\right| \leq \exp\left(-\frac{c_1}2t^2\sum_{|t|\log |t|<q_n<t^2} \frac{1}{q_n^2} + O\left(t^3\sum_{|t|\log |t|<q_n<t^2} \frac{1}{q_n^3}\right)\right).$$
The asymptotic for $\mathcal{A}(x)$ now implies that 
$$ \sum_{|t|\log |t|<q_n<t^2} \frac{1}{q_n^2}\gg \frac{(\log |t|)^{A-2}}{|t|} \text{ and } \sum_{|t|\log |t|<q_n<t^2} \frac{1}{q_n^3}\ll \frac{(\log |t|)^{A-3}}{t^2},$$
and so
$$\ex(e^{itH_{\Y}(Q)}) \ll \exp\left(-C_0 |t|(\log |t|)^{A-2}\right), $$
for some positive constant $C_0$. By letting $Q\to \infty$ in \eqref{ApproxChar} one deduces that the same estimate holds for $\varphi(t)$, which completes the proof.
\end{proof}

We now deduce Theorem \ref{LimitLaw}.
\begin{proof}[Proof of Theorem \ref{LimitLaw}]
First, note that $\ex(e^{itH_{\X}(Q)})$ converges uniformly to $\varphi(t)$ on any finite interval by \eqref{ConvChar}. This shows that $\varphi(t)$ is continuous at $t=0$ since $\ex(e^{itH_{\X}(Q)})$  is continuous at $t=0$. Therefore, it follows from L\'evy's continuity theorem that $\varphi(t)$ is the characteristic function of some random variable $\mathbb{W}$, and that $H_{\X}(Q)$ converges in distribution to $\mathbb{W}$. Furthermore, since $\varphi(t)= \ex(e^{it\mathbb{W}})$ is rapidly decreasing by \eqref{DecayChar}, then $\mathbb{W}$ is absolutely continuous and has a bounded probability density function on $\mathbb{R}$. Therefore, it follows from the Berry-Esseen inequality (see \S II.7.6 of \cite{Te}) that for all $T>0$ we have
$$  \sup_{V\in \mathbb{R}} \big|\Phi_Q(\X, V)- \pr(\mathbb{W}>V)\big| \ll \frac{1}{T} +\int_{-T}^T \frac{|\ex(e^{itH_{\X}(Q)})-\varphi(t)|}{t} dt.$$ 
Finally, choosing $T= \min\big(L/(10H(Q)), Q^{1/3}(\log Q)^{-(A-1)/3}\big)$ and using \eqref{ConvChar} gives
$$ \sup_{V\in \mathbb{R}} \big|\Phi_Q(\X, V)- \pr(\mathbb{W}>V)\big| \ll \frac{1}{T}+ \frac{T^2(\log Q)^{A-1}}{Q}\ll \frac{1}{T}\ll \frac{H(Q)}{L}+  \frac{(\log Q)^{(A-1)/3}}{Q^{1/3}},$$
as desired.
\end{proof}

 
We now prove the analogue of Theorem \ref{Main*} in the case of a sum of approximately independent random variables. Since the proof is similar, we shall only indicate the main differences that occur.

\begin{thm}\label{thm: Main}
Let $Q$ be large. Suppose that the random variables $\{\X(n)\}_{q_n\leq Q}$  satisfy AIH$(L,Q)$, approximated by the independent random variables $\{\Yn\}_{q_n\leq Q}$, with $(\log Q)^{2A^2} \leq L \leq \mathcal A(Q)$. Moreover suppose that the $\{\Yn)\}_{q_n\leq Q}$ satisfy CRVH$(\Y)$. Suppose that $V$ is in the range
\begin{equation}\label{DefRLQ}
c\leq V\leq R(L, Q):= \begin{cases} \frac{\alpha}{A} (\log L)^{A} +C_{\mathcal A}- \big( \alpha A \log\log Q +\Theta\big) (\log L)^{A-1}  & \text { if } A>0, \\
  \alpha \log\log L +C_{\mathcal A} -(\alpha \log \log\log Q +\Theta)/\log L & \text { if } A=0, \\
 \end{cases}
 \end{equation}
 where $\Theta>0$ and $c>1$ are suitably large constants.  Then
 $$ \Phi_Q(V)= \exp\left(-\alpha e^{-\eta_{\Y}-1} Z (\log Z)^{A-1} \left(1+O\left(\sqrt{\E(Z)}\right)\right)\right),$$
 where $\E(Z)$ is defined in \eqref{EpsilonS}.
\end{thm}

\begin{proof} 
We modify the proof of  Theorem \ref{Main*}. The key issue is that we wish to employ Proposition \ref{LaplaceTransformQ} which requires that $s\ll L/H(Q)$. 

Now  $ L\leq \mathcal A(Q)\sim \alpha Q(\log Q)^{A-1}$, and 
\[
H(Q)= \sum_{q_n\leq Q}\frac{1}{q_n}\sim  \begin{cases} \frac{\alpha}{A} (\log Q)^{A} & \text{ if } A>0, \\
\alpha \log\log Q & \text{ if } A=0, \end{cases}
\]
by Lemma \ref{LogSum}. Therefore $\tfrac L{H(Q)} \leq (A+o(1)) \tfrac Q{\log Q}$, and so
\begin{equation}\label{RangeS}
\min\left\{\frac{Q}{\log Q}, \frac{L}{10H(Q)}\right\} \gg \frac{L}{H(Q)} \gg \begin{cases} \frac{L}{(\log Q)^{A}} \text{ if } A>0, \\ 
\frac{L}{\log\log  Q} \text{ if } A=0. \end{cases}
\end{equation}

By hypothesis we have $V\leq R(L,Q)$, so that $H(Z)\leq R(L,Q)+O(1/Z)$. We can again use Lemma \ref{LogSum}  to evaluate $H(Z)$, and using the definition of $R(L, Q)$ in \eqref{DefRLQ} we deduce that 
if $A>0$ then $Z\ll L/(\log Q)^{A}$;
and if $A=0$ then $Z\ll L/\log\log Q$.   We deduce from \eqref{RangeS} that, in either case,
\[
s(V)\ll Z\ll \min\left\{\frac{Q}{\log Q}, \frac{L}{10H(Q)}\right\} ,
\]
and we can therefore apply  Proposition \ref{Cumulant} for any $s\asymp s(V)$, and we complete the proof here, exactly as in the proof
of Theorem \ref{Main*}.
\end{proof}

\begin{proof}[Proof of Theorem \ref{cor: Main2}]
We deduce Theorem \ref{cor: Main2} from Theorem \ref{thm: Main}, exactly as we 
 deduce Theorem \ref{cor: Main2*} from Theorem \ref{Main*}. The only difference is determining the range for $W$ which follows from the analogous calculation.
 \end{proof}





\begin{ack} 
Thanks to Nadia Sidorova for useful conversations and one of the referees for a very careful reading of the paper which highlighted some necessary minor corrections.
\end{ack} 

\bibliographystyle{plain}

\end{document}